%% file: main.tex
\documentclass{article}
\usepackage[T1]{fontenc}
\usepackage{files/arxiv}
\usepackage{amsmath, amsfonts, amsthm, amssymb}
\usepackage[colorlinks=true, allcolors=blue]{hyperref} 
\usepackage{graphicx}
\usepackage{algorithm, algorithmic}
\usepackage{subfigure}
\usepackage{xcolor}
\usepackage{setspace}
\input{files/definitions.tex}
\definecolor{gri}{RGB}{55,55,55}
\newcommand{\gray}[1]{\textcolor{gri}{{#1}}}
\newtheorem{remark}{Remark}
\newtheorem{theorem}{Theorem}
\newtheorem{lemma}{Lemma}
\newtheorem{corollary}{Corollary}
\usepackage{tikz,xcolor,hyperref}

\definecolor{lime}{HTML}{A6CE39}
\DeclareRobustCommand{\orcidicon}{
	\begin{tikzpicture}
	\draw[lime, fill=lime] (0,0) 
	circle [radius=0.16] 
	node[white] {{\fontfamily{qag}\selectfont \tiny ID}};
	\draw[white, fill=white] (-0.0625,0.095) 
	circle [radius=0.007];
	\end{tikzpicture}
	\hspace{-2mm}
}
\foreach \x in {A, ..., Z}{\expandafter\xdef\csname orcid\x\endcsname{\noexpand\href{https://orcid.org/\csname orcidauthor\x\endcsname}
			{\noexpand\orcidicon}}
}
\title{M-IHS: An Accelerated Randomized Preconditioning Method Avoiding Costly Matrix Decompositions}

\author{ \orcidA{}\hspace{1mm}Ibrahim Kurban Ozaslan\thanks{This author's studies have been supported by the Scientific and Technological
Research Council of Turkey (T\"UB\.ITAK) B\.IDEB 2210-A Scholarship Program.} \\
	Department of Electrical and Electronics Engineering\\
	Bilkent University\\
	Ankara, Turkey 06800 \\
	\texttt{ozaslan@usc.edu} \\
	\And
	\orcidB{}\hspace{1mm}Mert Pilanci\thanks{This author's work was partially supported by the National Science Foundation under grant IIS-1838179.} \\
	Department of Electrical Engineering\\
	Stanford University\\
	CA, USA \\
	\texttt{pilanci@stanford.edu} \\
	\And
    \orcidC{}\hspace{1mm}Orhan Arikan\\
	Department of Electrical and Electronics Engineering\\
	Bilkent University\\
	Ankara, Turkey 06800 \\
	\texttt{oarikan@ee.bilkent.edu.tr} \\
}

\renewcommand{\shorttitle}{Momentum Iterative Hessian Sketch}

\hypersetup{
pdftitle={Momentum Iterative Hessian Sketch},
pdfsubject={q-bio.NC, q-bio.QM},
pdfauthor={İbrahim. K. Ozaslan,Mert Pilanci, Orhan Arikan},
pdfkeywords={Tikhonov regularization, ridge regression, random projection, randomized preconditioning, acceleration},
}

\begin{document}
\maketitle

\begin{abstract}
\input{sections/abstract}
\end{abstract}

% keywords can be removed

\section{Introduction}\label{sec:intro}
\input{sections/intro.tex}
\section{Contributions}\label{sec:contr}
\input{sections/contribution.tex}
\section{The proposed M-IHS solvers for the regularized LS problems}\label{sec:prop}
\input{sections/algo_and_proof.tex}
\section{Numerical Experiments and Comparisons} \label{sec:numeric}
\input{sections/numeric.tex}

\section{Conclusions} \label{sec:conc}
\input{sections/conclusion.tex}

\bibliographystyle{unsrt}
\bibliography{files/references}
\appendix
\section{Discussion on the Proposed Error Upper Bound for Iterations of Primal Dual Algorithms in \cite{ref:acc_ihs}}\label{sec:app}
\input{sections/inaccuracies.tex}

\end{document}

%% file: sections/abstract.tex
Momentum Iterative Hessian Sketch (\ff{M-IHS}) techniques, a group of solvers for large scale regularized linear Least Squares (LS) problems, are proposed and analyzed in detail. Proposed \ff{M-IHS} techniques are obtained by incorporating the Heavy Ball Acceleration into the Iterative Hessian Sketch algorithm and they provide significant improvements over the randomized preconditioning techniques. By using approximate solvers along with the iterations, the proposed techniques are capable of avoiding all matrix decompositions and inversions, which is one of the main advantages over the alternative solvers such as the Blendenpik and the LSRN. Similar to the Chebyshev semi-iterations, the \ff{M-IHS} variants do not use any inner products and eliminate the corresponding synchronization steps in hierarchical or distributed memory systems, yet the \ff{M-IHS} converges faster than the Chebyshev Semi-iteration based solvers. Lower bounds on the required sketch size for various randomized distributions are established through the error analyses. Unlike the previously proposed approaches to produce a solution approximation, the proposed \ff{M-IHS} techniques can use sketch sizes that are proportional to the statistical dimension which is always smaller than the rank of the coefficient matrix. The relative computational saving gets more significant as the regularization parameter or the singular value decay rate of the coefficient matrix increase.  

%% file: sections/intro.tex
We are presenting a group of solvers, named as Momentum Iterative Hessian Sketch, \ff{M-IHS}, that is designed for solving large scale linear system of equations in the form of
\begin{equation}
    Ax_{0} + \omega = b, \label{eq:setup}
\end{equation}
where $A\in\realss{n}{d}$ is the given data or coefficient matrix, $b$ is the given measurement vector contaminated by the noise or computation/discretization error $\omega$, and $x_{0}$ is the vector desired to be recovered. Due to contaminated measurements, solutions can differ according to the constraints imposed on the problem. In this article, we are particularly interested in the $\ell2$-norm regularized Least Squares (LS) solution:
\begin{align}
    x^{*}= &\swrite{ argmin }{x} \underbrace{\frac{1}{2}\lscost{Ax - b}{2}^2 + \frac{\lambda}{2}\lscost{x}{2}^2}_{f(x)}, \label{eq:intro}
\end{align}
which is known as the Tikhonov Regularization or the Ridge Regression \cite{ref:bjorck}. The problem in \eqref{eq:intro} frequently arises in various large scale applications of science and engineering. For example, such regularized problems appear in the discretization of Fredholm Integral Equations of the first kind \cite{ref:hansen_book_2}. In those cases, the data matrix might be ill conditioned and the linear system can be either over-determined or square. When the system is under-determined, although sparse solutions are recently popularized by Compressed Sensing \cite{ref:compressed_sensing}, the least norm solutions occupy a fundamental place in statistics applications such as the Support Vector Machines \cite{ref:dual_rp, ref:dual_drineas}. Solutions to the problem in both regimes, i.e, $n\geq d$ and $n<d$, are often required as intermediate steps of rather complicated algorithms such as the Interior Point and the ADMM that are widely used in machine learning and image processing applications \cite{ref:admm, ref:bubeck, ref:newton_sketch}.

Throughout the manuscript, it is assumed that a proper estimate for the regularization parameter $\lambda$ is available. In the absence of such an estimate, risk estimators such as the Discrepancy Principle, Unbiased Prediction Risk Estimate, Stein's Unbiased Risk Estimate and Generalized Cross Validation can be directly used to obtain an estimate for the regularization parameter in the moderate size problems \cite{ref:vogel}. In large scale problems, these risk estimators can be adapted for the lower dimensional sub-problems that arise during the iteration of the first order iterative solvers \cite{ref:siam_review_hybrid}. A hybrid scheme that adaptively selects the regularization parameter along with the iterations is also suitable for the proposed \ff{M-IHS} solvers. Indeed, we have developed such a technique that can be found in Chapter 4 of \cite{ref:iko_ms}, but we will present it in a separate manuscript due to page length constraints.

The regularized solution in \eqref{eq:intro} can be obtained by using \textit{direct} methods such as the Cholesky decomposition for square $A$, or the QR decomposition for rectangular $A$. However, $O(nd\min(n,d))$ computational complexity of any full matrix decomposition becomes prohibitively large as the dimensions increase. For large scale problems, linear dependence on both dimensions might seem acceptable and can be realizable by using the first order iterative solvers that are based on the Krylov Subspaces \cite{ref:bjorck}. These methods require only a few matrix-vector and vector-vector multiplications for each iteration, but the number of iterations that is needed to reach high level of accuracies is highly sensitive to the condition number of the coefficient matrix. For the problem given in \eqref{eq:intro}, the convergence rate of the first order iterative solvers based on the Krylov subspace iterations including Conjugate Gradient (CGLS), LSQR, LSMR, Chebyshev Semi-iterative (CS) technique and many others is characterized by the following inequality:
\begin{equation*}
    \lscost{x^i - x^*}{2} \leq \left(\frac{\sqrt{\kappa(A^TA+\lambda I_d)} - 1}{\sqrt{\kappa(A^TA+\lambda I_d)} +1}\right)^i \lscost{x^1 - x^*}{2}, \ 1 < i,
\end{equation*}
where $x^*$ is the optimal solution of \eqref{eq:intro}, $x^1$ is the initial guess, $x^i$ is the $i$-th iterate of the solver and the condition number $\kappa(\cdot)$ is defined as the ratio of the largest singular value to the smallest singular value of its argument \cite{ref:krylov_convergence}. Since for ill conditioned matrices $\kappa(A^TA+\lambda I_d)$ can be large, the rate of convergence may be extremely slow.

The computational complexity of the Krylov Subspace-based iterative solvers is $O(nd)$ for each iteration, which is significantly less than $O(nd\min(n,d))$ if the number of iterations can be significantly fewer than $\min(n,d)$. However, in applications such as big data where $A$ is very large dimensional, the computational complexity is not the only metric for feasibility of the algorithms. For example, if the coefficient matrix is too large to fit in a single working memory and it could be merely stored in a number of distributed computational nodes, then at least two distributed computations of matrix-vector multiplications are required at each iteration of algorithms such as the CGLS or the LSQR \cite{ref:distributed_pilanci, ref:distributed_1}. Therefore the number of iterations should also be counted as an important metric to measure the overall complexity of an algorithm. One way to reduce the number of iterations in the iterative solvers is to use preconditioning to transform an ill conditioned problem to a well conditioned one \cite{ref:templates}. In the deterministic settings, finding a low-cost and effective preconditioning matrix is still a challenging task unless the coefficent matrix has a particular structure \cite{ref:precond}. 

In addition to the number of iterations, the number of inner products in each iteration also plays an important role in the overall complexity. Each inner product calculation constitutes a synchronization step in parallel computing and therefore is undesirable for distributed or hierarchical memory systems \cite{ref:templates}. The CS technique can be preferred in this kind of applications, since it does not use any inner products and therefore eliminates some of the synchronization steps that are required by the techniques such as the CG or GMRES. However, the CS requires prior information about the ellipsoid that contains all the eigenvalues of $A$, which is typically not available in practice \cite{ref:par_cheby}.

These aforementioned drawbacks of the direct and the iterative methods can be remedied by using the Random Projection (RP) techniques \cite{ref:jl_lemma, ref:randnla}. These techniques are capable of both reducing the dimensions and bounding the number of iterations with statistical guarantees, while they are quite convenient for parallel and distributed computations. The development and the applications of the RP based algorithms can be found in \cite{ref:mahoney_book, ref:w14_sketching_asatool} and references therein. There are two main approaches for the applications of the RP based techniques to the regularized LS problem in \eqref{eq:intro}. In the first approach, that is referred to as the \textit{classical sketching}, the coefficient matrix $A$ and the measurement vector $b$ are projected down onto a lower dimensional ($m\ll n$)\footnote{Without loss of generality we can assume the linear system is over-determined, if it is not, then we can take a dual of the problem to obtain an over-determined problem as examined later.} subspace by using a randomly constructed \textit{sketching matrix} $S\in\realss{m}{n}$, to obtain efficiently an $\zeta$-optimal solution with high probability for the \textit{cost approximation} \cite{ref:faster_ls, ref:pilanci1}:
\begin{equation}
    \widetilde{x} = \swrite{argmin}{x}\frac{1}{2}\lscost{SAx - Sb}{2}^2 + \frac{\lambda}{2}\lscost{x}{2}^2, \mbox{ such that }  f(\widetilde{x})\leq (1+\zeta)f(x^*). \label{eq:cost_approx}
\end{equation}
For both the sparse and dense systems, in \cite{ref:acw16_sharper} the best known lower bounds on the sketch size for obtaining an $\zeta$-optimal cost approximation have been derived showing that the sketch size can be chosen proportional to the \textit{statistical dimension} $\sd(A) = \tr{A(A^TA+\lambda I_d)^{-1}A^T}$. Although the cost approximation is sufficient for many machine learning problems, the \textit{solution approximation} which aims to produce solutions that are close to the optimal solution is a more preferable metric for the problems arisen from, for example, discretization of Fredholm integrals \cite{ref:bjorck, ref:vogel}. However, as shown in \cite{ref:ihs}, the classical sketching is sub-optimal in terms of the minimum sketch size for obtaining a solution approximation. 

In the second approach of \textit{randomized preconditioning}, by iteratively solving a number of low dimensional sub-problems constituted by $(SA, \nabla f(x^i))$ pairs, algorithms with reasonable sketch sizes obtain an $\eta$-optimal solution approximation $\widehat{x}$ such that
\begin{equation}
    \lscost{\widehat{x} - x^*}{W} \leq \eta \lscost{x^*}{W}, \label{eq:sol_approx}
\end{equation}
where $W$ is a positive definite weight matrix. In \cite{ref:rokhlin}, RP techniques have been proposed to construct a preconditioning matrix for CG-like algorithms by using the inverse of $R$ factor in the QR decomposition of the \textit{sketched matrix} $SA$. Later, implementation of similar ideas resulted in Blendenpik and LSRN which have been shown to be faster than some of the deterministic solvers of LAPACK \cite{ref:blendenpik, ref:lsrn}. To solve the preconditioned problems, as opposed to the Blendenpik which uses the LSQR, the LSRN uses the CS technique for parallelization purposes and deduce the prior information about the eigenvalues based on the results of the random matrix theory. The main drawback of the LSRN and the Blendenpik is that regardless of the desired accuracy $\eta$, one has to pay the whole cost, $O(md^2)$, of a full $m\times d$ dimensional matrix decomposition, which is the dominant term in the computational complexity of these algorithms. Iterative Hessian Sketch (IHS) proposed in \cite{ref:ihs} enables the use of the sketched Hessian as preconditioning matrix in the Gradient Descent method \cite{ref:acc_ihs} thereby providing a reduction in the dominant complexity term $O(md^2)$ to $O(md)$. For the IHS, instead of computing a full decomposition or an inversion, a linear system can be approximately solved for a pre-determined tolerance. The ability to use this \textit{inexact} approach becomes more important in the large scale inverse problems such as 3D imaging \cite{ref:3d_image} where even the decomposition of $m\times d$-dimensional sketched matrix is infeasible to compute. By using the preconditioning idea of the IHS in the CG technique, Accelerated IHS (A-IHS) has been proposed in \cite{ref:acc_ihs}. Lastly, in \cite{ref:mihs}, it has been showed that if the linear system is strongly over-determined, then the momentum parameters of the Heavy Ball Method can be robustly estimated by using Marchenko Pastur Law (MPL) \cite{ref:edelman}. This analysis results in a prototype solver \ff{M-IHS} that we study here in detail.

The statistical lower bounds obtained in the current literature suggest that the sketch size in randomized preconditioning algorithms can be chosen proportional to the rank of the problem, which can be significantly larger than the statistical dimension. Although, some lower bounds on the sketch size that are proportional to the statistical dimension have been obtained in Kernel Ridge Regression \cite{ref:avron_kernel}, it is not available for the regularized LS problem for obtaining a solution approximation given in \eqref{eq:sol_approx}.

%% file: sections/contribution.tex
In this article, we propose a group of random projection based iterative solvers for large scale regularized LS problems. As shown by detailed analyses of their convergence behaviours, the proposed \ff{M-IHS} variants can be used for any dimension regimes with significant computational savings if the statistical dimension of the problem is sufficiently smaller than at least one size of the coefficient matrix. Our guarantees, presented in Theorem \ref{lemma:mihs_conv} and Corollary \ref{corollary:iter}, are based on the solution approximation metric given in \eqref{eq:sol_approx} as opposed to the results obtained for cost approximation metric given in \eqref{eq:cost_approx}. In Lemma \ref{corollary:sketch}, we improved the known lower bounds on the sketch size of various randomized distribution for obtaining a pre-determined convergence rate with a constant probability. These guarantees can be readily extended to any other sketching types by using the \textit{Approximate Matrix Multiplication} (AMM) property defined in \cite{ref:cnw15_optimal_stable_rank}. When tighter bounds for the AMM property will be available in the future, the bounds derived in this work can be automatically improved as well. Although our bounds for the dense sketch matrices such as Subgaussian or Randomized Orthonormal Systems (ROS) are the same as in \cite{ref:acw16_sharper}, we gained slightly better results for the sparse sketching matrices. Additionally, we provide some approximate bounds for the sketch size and the rate of convergence in Corollary \ref{corollary:empirical} which is remarkably tight as demonstrated through numerical experiments. Lastly, in Algorithm \ref{algo:aab_solver}, we extend the idea of LSQR into the linear problems in the form of $A^TAx=b$ which we need to solve during the iterations of all proposed \textit{Inexact} \ff{M-IHS} variants and of the Newton Sketch \cite{ref:newton_sketch}. Similar to the stability advantage of the LSQR over the CGLS technique \cite{ref:lsqr}, the proposed method solves the system in the above form without squaring the condition number as opposed to the techniques such as the symmetric CG and the symmetric Lanczos techniques. In the following link, implementations of the proposed solvers in \textit{MATLAB} together with the codes that generate the figures in the article can be found: \url{https://github.com/ibrahimkurban/M-IHS}.

%% file: sections/algo_and_proof.tex
The naive IHS algorithm approximates the Hessian in the Newton method to gain computational savings while solving the Newton sub-systems, and it iteratively minimizes the quadratic objective function given in \eqref{eq:intro} by performing the following updates:
\begin{align}
    x^{i+1} &= \swrite{argmin }{x\in\reals{d}} \lscost{S_iA(x - x^i)}{2}^2 + \lambda\lscost{x-x^i}{2}^2 + 2\langle \nabla f(x^i),\ x \rangle.\label{eq:ihs}
\end{align}
Here, we propose two important innovations over the naive IHS technique, which provide further computational efficiency and improved convergence rate. We realized these two goals by incorporating the Heavy Ball Acceleration \cite{ref:polyak} into the iterations in \eqref{eq:ihs} and propose the following Momentum-IHS (\ff{M-IHS}) updates:
\begin{align}
    \Delta x^i &= \swrite{argmin }{x\in\reals{d}} \lscost{SAx}{2}^2 + \lambda\lscost{x}{2}^2 + 2\left\langle \nabla f(x^i), \ x\right\rangle, \label{eq:mihs1}\\
    x^{i+1}&= x^i + \alpha_i\Delta x^i + \beta_i \left(x^i - x^{i-1}\right), \nonumber
\end{align}
where the same sketching matrix $S$ is used for all iterations. For a properly chosen momentum parameters $\alpha_i$ and $\beta_i$, there is no need to change the sketching matrix in the above iterations unlike the IHS technique. Moreover, the optimal fixed momentum parameters $\alpha$ and $\beta$, that maximize the convergence rate, can be estimated by using the random matrix theory as completely independent of the spectral properties of the coefficient matrix $A$ \cite{ref:mihs}. Here, the linear system is assumed to be strongly over-determined, i.e., $n\gg d$. By using the dual formulation, the theory can be straightforwardly extended to the strongly under-determined case of $d\gg n$ as well \cite{ref:boyd_cvx}. A dual of the problem in \eqref{eq:intro} is
\begin{equation}
    \nu^* = \swrite{argmin }{\nu\in\reals{n}} \underbrace{\frac{1}{2}\lscost{A^T\nu}{2}^2 + \frac{\lambda}{2}\lscost{\nu}{2}^2 - \langle b,\ \nu\rangle}_{g(\nu)}, \label{eq:intro_dual}
\end{equation}
and the relation between the solutions of the primal and dual problem is
\begin{equation}
    \nu^* = (b - Ax^*)/\lambda \Longleftrightarrow x^* = A^T\nu^*. \label{eq:dual_relation}
\end{equation}
The corresponding \ff{M-IHS} iterations for the dual problem are:
\begin{align}
    \Delta \nu^i &= \swrite{argmin }{\nu\in\reals{n}}\lscost{SA^T\nu}{2}^2 +\lambda\lscost{\nu}{2}^2 + 2\left\langle \nabla g(\nu^i),\ \nu\right\rangle, \label{eq:dual_mihs1}\\
    \nu^{i+1} &= \nu^{i} + \alpha \Delta\nu^i + \beta\left(\nu^i - \nu^{i-1}\right),\nonumber
\end{align}
and the solution of the primal problem can be obtained through the relation in \eqref{eq:dual_relation}. We refer to this algorithm as the \ff{Dual M-IHS}. The convergence rate of the \ff{M-IHS} and the \ff{Dual M-IHS} solvers together with the optimal fixed momentum parameters are stated in the Theorem \ref{lemma:mihs_conv} below.
\begin{theorem}\label{lemma:mihs_conv}
Let $A$ and $b$ be the given data in \eqref{eq:setup} with singular values $\sigma_i$ in descending order $1\leq i\leq\min(n,d)$, $x^*\in\reals{d}$ and $\nu^*\in\reals{n}$ are as in \eqref{eq:intro} and \eqref{eq:intro_dual}, respectively. Let $U_1 \in\realss{\max(n,d)}{\min(n,d)}$ consists of the first $n$ rows of an orthogonal basis for $[A^T\ \sqrt{\lambda}I_d]^T$ if the problem is  over-determined, and consists of the first $d$ rows of an orthogonal basis for $[A \ \ \sqrt{\lambda}I_n]^T$ if the problem is under-determined. Let the sketching matrix $S\in\realss{m}{\max(n,d)}$ be drawn from a distribution $\D$ such that
\begin{equation}
    \prob{S\sim\D}{\lscost{U_1^TS^TSU_1 - U_1^TU_1}{2} \geq \epsilon} < \delta, \quad \epsilon \in(0,1) \label{condition:1}.
\end{equation}
Then, the {\ff{M-IHS}} applied on \eqref{eq:intro} and the {\ff{Dual M-IHS}} applied on \eqref{eq:intro_dual} with the following momentum parameters
\begin{equation*}
    \beta^* = \left(\frac{\epsilon}{1+\sqrt{1-\epsilon^2}}\right)^2, \quad\quad \alpha^* = \left(1 - \beta^*\right)\sqrt{1 - \epsilon^2},
\end{equation*}
converge to the optimal solutions, $x^*$ and $\nu^*$, respectively, at the following rate with a probability of at least $(1-\delta)$:
\begin{align*}
    \lscost{x^{i+1} - x^{*}}{D^{-1}} \leq \frac{\epsilon}{1 + \sqrt{1 - \epsilon^2}} \lscost{x^i - x^*}{D^{-1}}, \\ 
    \lscost{\nu^{i+1} - \nu^{*}}{D^{-1}} \leq \frac{\epsilon}{1 + \sqrt{1 - \epsilon^2}} \lscost{\nu^i - \nu^*}{D^{-1}},
\end{align*}
where $D^{-1}$ is the diagonal matrix whose diagonal entries are $\sqrt{\sigma_i^2 + \lambda}$. 
\end{theorem}
\begin{proof}\label{proof:lemma_conv}
\input{proofs/proof_1.tex}
\end{proof}
Note that, Theorem \ref{lemma:mihs_conv} is also valid for the un-regularized problems if, instead of \eqref{condition:1}, the following condition is satisfied
\begin{equation*}
    \prob{S\sim\D}{\lscost{U^TS^TSU - I_d}{2} \geq \epsilon} < \delta, \quad \epsilon \in(0,1).
\end{equation*}
When the necessary conditions are met, the number of iterations needed for both algorithms to reach a certain level of accuracy is stated in the following corollary.
\begin{corollary}\label{corollary:iter}
For some $\epsilon\in(0, 1/2)$ and arbitrary $\eta$, if the sketching matrix meets the condition in \eqref{condition:1} and the fixed momentum parameters are chosen as in Theorem \ref{lemma:mihs_conv}, then the number of iterations for the {\ff{M-IHS}} and the {\ff{Dual M-IHS}} to obtain an $\eta$-optimal solution approximation in $\ell2$-norm is upper bounded by
$$N = \left\lceil \frac{\log(\eta)\log(C)}{\log(\epsilon) - \log(1 + \sqrt{1 - \epsilon^2})}\right\rceil$$
where the constant $C$, that is defined as $C = \sqrt{\kappa(A^TA+\lambda I_d)}$ for the {\ff{M-IHS}} and $C = \kappa(A)\sqrt{\kappa(AA^T+\lambda I_n)}$ for the {\ff{Dual M-IHS}}, can be removed if the semi-norm in Theorem \ref{lemma:mihs_conv} is used as the solution approximation metric instead of the $\ell2$ norm.
\end{corollary}
Corollary \ref{corollary:iter} is an immediate result of Theorem \ref{lemma:mihs_conv}. To satisfy the condition in \eqref{condition:1}, a set of cases for the sketching matrix $S$ are given in Lemma \ref{corollary:sketch}.
\begin{lemma} \label{corollary:sketch}
If the sketching matrix $S$ is chosen in one of the following cases, the condition in \eqref{condition:1} of Theorem \ref{lemma:mihs_conv} is satisfied.
\renewcommand{\theenumi}{\roman{enumi}}%
\begin{enumerate}
    \item\label{item1} $S$ is a Sparse Subspace Embedding \cite{ref:w14_sketching_asatool} with single nonzero element in each column, with a sketch size
    $$m =\Omega\left( \mathbf{sd}_\lambda(A)^2/(\epsilon^2\delta)\right)$$
    where $\Omega(\cdot)$ notation is defined as $a(n) = \Omega(b(n))$, if there exists two integers $k$ and $n_0$ such that $\forall n > n_0$, $a(n)\geq k\cdot b(n)$. For this case, $SA$ is computable in $O(\mathbf{nnz}(A))$ operations.
    \item\label{item2} $S$ is a Sparse Subspace Embedding with 
    $$s = \Omega(\log_\alpha(\sd(A)/\delta)/\epsilon)$$
    non-zero elements in each column where $\alpha> 2$, $\delta<1/2$, $\epsilon<1/2$, \cite{ref:kn14_sparser_jl, ref:nn13_osnap}, with a sketch size    
    $$m = \Omega(\alpha\cdot\sd(A)\log(\sd(A)/\delta)/\epsilon^2).$$
    For this case, $SA$ is computable in $O(s\cdot\mathbf{nnz}(A))$ operations.
    \item\label{item3} $S$ is a SRHT sketching matrix \cite{ref:cnw15_optimal_stable_rank, ref:ihs} with a sketch size
    $$m= \Omega\left(\left(\mathbf{sd}_\lambda(A)+\log(1/\epsilon\delta)\log(\mathbf{sd}_\lambda(A)/\delta)\right)/\epsilon^{2}\right).$$
    For this case, $SA$ is computable in $O(nd\log(m))$ operations.
    \item\label{item4} $S$ is a Sub-Gaussian sketching matrix \cite{ref:pilanci1, ref:cnw15_optimal_stable_rank} with a sketch size
    $$m = \Omega(\sd(A)/\epsilon^2).$$ 
    For this case, $SA$ is computable in $O(ndm)$ operations.
\end{enumerate}
\end{lemma}
\begin{proof}\label{proof:cor_sketch}
\input{proofs/proof_2.tex}
\end{proof}
Lemma \ref{corollary:sketch} suggests that in order to satisfy the condition in Theorem \ref{lemma:mihs_conv}, the sketch size can be chosen proportional to the statistical dimension of the coefficient matrix which can be considerably smaller than its rank. Moreover, to obtain a solution approximation, the second condition in Lemma 11 of \cite{ref:acw16_sharper} is not a requirement, hence we obtained slightly better results for the sparse subspace embeddings in the cases of $(i)$ and $(ii)$ of Lemma \ref{corollary:sketch}. In the following corollary, we obtained substantially simplified empirical versions of the convergence rate, momentum parameters and required sketch size by using the MPL and approximating the filtering coefficients of Tikhonov regularization with binary coefficients. Corollary \ref{corollary:empirical} suggests that the ratio between the statistical dimension and the sketch size determines the convergence rate of the proposed algorithms, which interestingly seems valid even for the sketch matrices with a single non-zero element in each column.
\begin{corollary}\label{corollary:empirical}
If the entries of the sketching matrix are independent, zero mean, unit variance with bounded higher order moments, and the Truncated SVD regularization with truncation parameter $\lceil \sd(A) \rceil$ is used, then the {\ff{M-IHS}} and the {\ff{Dual M-IHS}} with the following momentum parameters
\begin{equation*}
    \beta = \frac{\sd(A)}{m}, \quad\quad \alpha = (1- \beta)^2
\end{equation*}
will converge to the optimal solutions $x^*$ and $\nu^*$ respectively with a convergence rate of $\sqrt{\beta}$ as $m\to \infty$ while $\sd(A)/m$ remains constant. Any sketch size $m > \sd(A)$ can be chosen to obtain an $\eta$-optimal solution approximation in most $\frac{\log(\eta)}{\log(\sqrt{\beta})}$ iterations. 
\end{corollary}
\begin{proof}\label{proof:empirical}
\input{proofs/proof_3.tex}
\end{proof}%
%%%%%%%%%%%%%%%%%%%%%%%%%%%%%%
\input{demos/fig_convergence}%
%%%%%%%%%%%%%%%%%%%%%%%%%%%%%%%
Although the MPL provides bounds for the singular values of the sketching matrix $S$ in the asymptotic regime, i.e., as $m\to\infty$; these bounds become very good estimators of the actual bounds when $m$ takes finite values, as demonstrated in Figure \ref{fig:convergence}.  In Figure \ref{fig:convergence_1}, $A\in\realss{32768}{1000}$ with $\kappa(A) = 10^8$ was generated as described in Section \ref{sec:ex_setup}. In Figure \ref{fig:convergence_2}, $A\in\realss{24336}{1296}$ was generated by using \textit{sprand} command of MATLAB. We first created a sparse matrix with size $\Tilde{A}\in\realss{20}{6}$ and sparsity of $15\%$, then the final form was obtain by taking $A = \Tilde{A}^{\otimes 4}$ and deleting the all-zero rows. The final form of $A$ has a sparsity ratio of $0.1\%$ and the condition number of $\kappa(A) = 10^7$. The noise level was set to $1\%$ and the regularization parameter $\lambda$ that minimizes the error $\|x_0 - x(\lambda)\|$ was used in both experiments. The resulting statistical dimensions were $119$ and $410$, respectively. The rate of $\sqrt{\beta}$ in Corollary \ref{corollary:empirical} creates a remarkable fit to the numerical convergence rate of the \ff{M-IHS} variants when the momentum parameters given in Corollary \ref{corollary:empirical} are used even for the Tikhonov regularization. This is because the sigmoid-like filtering coefficients in the Tikhonov regularization can be thought of as the smoothed version of the binary coefficients in the TSVD solution and therefore the binary coefficients constitute a good approximation for the filtering coefficients of the Tikhonov regularization.

\begin{remark}
The momentum parameters given in Corollary \ref{corollary:empirical} maximizes the convergence rate when the statistical dimension is known. If $\sd(A)$ is overestimated and thus $\beta$ is chosen larger than the ratio $\sd(A)/m$ and $\alpha = (1-\beta)^2$, then the convergence rate is still $\sqrt{\beta}$ since all the dynamical systems in \eqref{eq:characteristic} will be still in the under-damped regime. An empirical algorithm to estimate $\sd(A)$ by using the Hutchinson-like estimators is detailed in Section \ref{sec:sd}.
\end{remark}

\subsection{Efficient M-IHS sub-solvers}\label{sec:efficent_sub}
In practice, the \ff{M-IHS} and the \ff{Dual M-IHS} eliminate the dominant term $O(md^2)$ in the complexity expression of well known solvers such as the Blendenpik and the LSRN by approximately solving the lower dimensional linear systems in \eqref{eq:mihs1} and \eqref{eq:dual_mihs1} avoiding matrix decompositions or inversions. This \textit{inexact sub-solver} approach provides a trade-off opportunity between the computational complexity and the convergence rate, that is highly desirable in very large dimensional problems. Unfortunately, such a trade-off is not possible for the Blendenpik and the LSRN techniques which require a full matrix decomposition of the sketched matrix. Inexact sub-solvers have been known to be a good heuristic way to create this trade-off and they are widely used in the algorithms that are based on the Newton Method to solve the large scale normal equations \cite{ref:nocedal}. In these inexact (or truncated) Newton Methods, inner iterations are terminated at the moment that the relative residual error is lower than an iteration-dependent threshold, named as the \textit{forcing terms} \cite{ref:forcing_term}. In the literature, there are various techniques to choose these forcing terms that guarantee a global convergence \cite{ref:inexact_survey}, but the number of iterations suggested by these techniques are significantly higher than the total number of iterations used in practice. Therefore, in this work the heuristic constant threshold $\epsilon_{sub}$, that checks the relative residual error of the linear system, is used \cite{ref:nocedal2}.

Efficient but approximate solutions to the sub-problems in \eqref{eq:mihs1} and \eqref{eq:dual_mihs1} can be obtained by Krylov Subspace based first order methods. However, LSQR-like solvers that are adapted for the normal equations would require computations of 4 matrix-vector multiplications per iteration. On the other hand, due to the explicit calculation of $(SA)^T(SA)z$, the symmetric CG, that would require only 2 matrix-vector multiplications, might be unstable for the ill-conditioned problems \cite{ref:lsqr}. Therefore, in Section \ref{sec:subsolver}, we propose a stable sub-solver, referred to as \ff{AAb\_Solver}, which is particularly designed for the problems in the form of $A^TAx=b$. The \ff{AAb\_Solver} is based on the Golub Kahan Bidiagonalization and it uses a similar approach that the LSQR uses on the LS problem. The inexact versions of the \ff{M-IHS} and the \ff{Dual M-IHS} that use \ff{AAb\_Solver} are given in Algorithm \ref{algo:mihs} and Algorithm \ref{algo:dual_mihs}, %%
\input{demos/algo_mihs.tex}%
where \ff{RP\_fun} represents the function that generates the desired sketched matrix such that $\expectt{}{S^TS} = I_m$ whose implementation details can be found in the relevant references in Lemma \ref{corollary:sketch}. Number of operations required at each step is stated at the right most column of the algorithms, where $C(\cdot)$ represents the complexity of constructing the sketching matrix as given in Lemma \ref{corollary:sketch}. Setting the forcing term $\epsilon_{sub}$, for instance, to 0.1 for all iterations is enough for the inexact \ff{M-IHS} variants to converge at the same rate $\sqrt{\beta}$ as the exact versions as demonstrated in Figure \ref{fig:convergence}.

\subsection{Two-stage sketching for the M-IHS variants}
Lemma \ref{corollary:sketch} suggests that if the statistical dimension is several times smaller than the dimensions of $A$, then for the M-IHS techniques, it is possible to choose a substantially smaller sketch size than $\min(n,d)$. If this is the case, then the quadratic objective functions in \eqref{eq:mihs1} and \eqref{eq:dual_mihs1} become strongly under-determined problems, which makes it possible to approximate the Hessian of the objective functions one more time by taking their convex dual as it has been done in the \ff{Dual M-IHS}. This approach is similar to the approach where the problems in \eqref{eq:mihs1} and \eqref{eq:dual_mihs1} are approximately solved by using the \ff{AAb\_Solver}, with an additional dimension reduction. As a result of two Hessian sketching, the linear sub-problem whose dimensions are reduced from both sides can be efficiently solved by the \ff{AAb\_Solver} for a pre-determined tolerance as before. For the details of this two-step approach, consider the following dual of the sub-problem in \eqref{eq:mihs1}
\begin{equation}
    z^{i,*} = \swrite{argmin }{z\in\reals{m}} \underbrace{\frac{1}{2}\lscost{A^TS^Tz + \nabla f(x^i)}{2}^2 + \frac{\lambda}{2}\lscost{z}{2}^2}_{h(z, x^i)},
\end{equation}
which is a strongly over-determined problem if $m\ll \min(n,d)$. Hence, it can be approximately solved by the \ff{M-IHS} updates as
\begin{align}
    \Delta z^{i,j} &= \swrite{argmin }{z\in\reals{m}} \lscost{WA^TS^Tz}{2}^2 + \lambda\lscost{z}{2}^2 + 2\left\langle \nabla_z h(z^{i,j}, x^i), \ z\right\rangle \label{eq:pd_mihs1},\\
    z^{i, j+1}&= z^{i,j} + \alpha_2\Delta z^{i,j} + \beta_2 \left(z^{i,j} - z^{i,j-1}\right). \nonumber
\end{align}
After $M$ iterations, the solution of \eqref{eq:mihs1} can be recovered by using the relation in \eqref{eq:dual_relation} as $\Delta x^i = (\nabla f(x^i) - A^TS^Tz^{i,M})/\lambda$. The same strategy can be applied on the sub-problem in \eqref{eq:dual_mihs1} by replacing $SA$ with $SA^T$ and $\nabla f(x^i)$ with $\nabla g(\nu^i)$. The resulting algorithms, referred to as \ff{Primal Dual M-IHS}, are given in Algorithm \ref{algo:pd_mihs_1} and Algorithm \ref{algo:pd_mihs_2}, respectively.%
\input{demos/algo_pd.tex}

The primal-dual idea presented here is first suggested by Zhang et al. in \cite{ref:acc_ihs}. They used the A-IHS technique to solve the sub-problems that arise during the iterations of the Accelerated Iterative Dual Random Projection (A-IDRP) which is a dual version of the A-IHS. However, since both of the A-IHS and the A-IDRP are based on the CG technique, the convergence rate of the proposed A-IHS, A-IDRP and the primal dual algorithm called as Accelerated Iterative  Primal Dual Sketch (A-IPDS) are all degraded in the LS problems with high condition numbers due to the instability issue of the symmetric CG technique \cite{ref:lsqr}. Even if the regularization is used, still the performance of the solvers proposed in \cite{ref:acc_ihs} are considerably deteriorated compared to the other randomized preconditioning techniques as shown in Section \ref{sec:numeric}. Further, applying the preconditioning idea of IHS to the stable techniques such as the LSQR that are adapted for the LS problem is not so efficient as the \ff{M-IHS} variants, because they require two preconditioning systems to be solved per iteration.

The computational saving when we apply a second dimension reduction as in the \ff{Primal Dual M-IHS} may not be significant due to the second gradient computations in \textit{Line 10} of the given algorithms, but the lower dimensional sub-problems that we obtain at the end of the second sketching can be used to estimate several parameters including the regularization parameter itself, if it is unknown. As we have shown in \cite{ref:iko_ms}, such two-stage sketching approach is particularly effective to estimate the unknown regularization parameter when the coefficient matrix is square or close to be square.  

The \ff{Primal Dual M-IHS} techniques are extension of the inexact schemes. Therefore, their convergence rates depend on their forcing terms that are used to stop the inner iterations \cite{ref:inexact_survey}. In \cite{ref:acc_ihs}, an upper bound for the error of the primal dual updates is proposed. However as it is detailed in Section \ref{sec:app}, there are several inaccuracies in the development of the bound. Therefore, finding a provably valid lower bound on the number of inner loop iterations, that guarantee a certain rate of convergence at the main loop, is still an open problem for the primal dual algorithms.
\subsection{Estimation of the statistical dimension}\label{sec:sd}
\input{sections/trace_estimator.tex}
\subsection{Complexity analyses of the proposed algorithms}\label{sec:complexity}
\input{sections/complexity2.tex}
\subsection{A solver for linear systems in the form of \texorpdfstring{$A^TAx = b$}{Lg}}\label{sec:subsolver}
\input{sections/aab_solver.tex}

%% file: proofs/proof_1.tex
In the following analysis we denote $A = U\Sigma V^T$ as the compact SVD with $U\in\realss{n}{r}$, $\Sigma = \diag(\sigma_1,\ldots,r)\in\realss{r}{r}$ and $V \in\realss{d}{r}$ where $r =\min(n,d)$. To prove the theorem for the \ff{M-IHS} and the \ff{Dual M-IHS}, we mainly combine the idea of \textit{partly exact} sketching, that is proposed in \cite{ref:acw16_sharper}, with the Lyapunov analysis, that we use in \cite{ref:mihs}. In parallel to \cite{ref:acw16_sharper}, we define the diagonal matrix $D:=(\Sigma^2 + \lambda I_{r})^{-1/2}$ and the \textit{partly exact} sketching matrix $S$ as:
\begin{equation*}
    \widehat{S} = \begin{bmatrix}S &\mathbf{0}\\\mathbf{0} & I_r\end{bmatrix}, \quad S\in\realss{m}{\max(n,d)}.
\end{equation*}
\paragraph{The proof for M-IHS} 
Let
\begin{equation*}
    \widehat{A} = \begin{bmatrix} U\Sigma D \\ \sqrt{\lambda} VD\end{bmatrix} = \begin{bmatrix} U_1 \\U_2\end{bmatrix}, \quad \widehat{A}^T\widehat{A} = I_d, \quad \widehat{b} = \begin{bmatrix}b\\\mathbf{0}\end{bmatrix},
\end{equation*}
so that $U_1$ is the first $n$ rows of an orthogonal basis for $[A^T\ \ \sqrt{\lambda} I_d ]^T$ as required by the condition in \eqref{condition:1} of the theorem. To simplify the Lyapunov analysis, the following LS problem will be used:
\begin{equation}
    y^* = \swrite{argmin }{y\in\reals{d}}\lscost{\widehat{A}y - \widehat{b}}{2}^2 \label{eq:equi1}
\end{equation}
which is equivalent to the problem in \eqref{eq:intro} due to the one-to-one mapping $\{\forall x(\lambda)\in\reals{d}\ | \ y^* = D^{-1}V^Tx(\lambda)\}$. For the problem in \eqref{eq:equi1}, the equivalent of the \ff{M-IHS} given in \eqref{eq:mihs1} is the following update:
\begin{align*}
    \Delta y^{i} &= \swrite{argmin}{y} \lscost{\widehat{S}\widehat{A}y}{2}^2 - 2\langle \widehat{A}^T(\widehat{b} - \widehat{A}y^i), \ y \rangle \\
    y^{i+1} &= y^i + \alpha\Delta y^i + \beta(y^i - y^{i-1})
\end{align*}
with sketched matrix
\begin{equation*}
     \quad \widehat{S}\widehat{A} = \begin{bmatrix} SU\Sigma D\\ \sqrt{\lambda}VD\end{bmatrix} = \begin{bmatrix} SU_1\\U_2\end{bmatrix}.
\end{equation*}
Thus, we can examine the following bipartite transformation to find out the convergence properties of the \ff{M-IHS}:
\begin{equation*}
    \begin{bmatrix}
    y^{i+1} - y^* \\
    y^{i} - y^*
    \end{bmatrix} = \underbrace{\begin{bmatrix}
    (1+\beta)I_d - \alpha (\widehat{A}^T\widehat{S}^T\widehat{S}\widehat{A})^{-1} & -\beta I_d\\
    I_d & \mathbf{0}
    \end{bmatrix}}_{T} \begin{bmatrix}
    y^{i} - y^* \\
    y^{i-1} - y^*
    \end{bmatrix}.
\end{equation*}
The contraction ratio of the transformation, which is determined by the eigenvalues, can be found analytically by converting the matrix $T$ into a block diagonal form through the following similarity transformation:
\begin{align*}
T &= P^{-1} \diag(T_1,\ldots, T_d) P,\text{ where } T_i := \begin{bmatrix}
    1 + \beta - \alpha \mu_i & \beta \\
    1 & 0
\end{bmatrix},\quad
P = \begin{bmatrix}
    \Psi & 0 \\
    0 & \Psi
\end{bmatrix} \Pi,\quad
\Pi_{i,j} = \left\{
\small\begin{array}{rl}
1 & \text{i \mbox{is odd} } j = i,\\
1 &  \text{i \mbox{is even} } j = r+i,\\
0 &  \text{otherwise,}
\end{array} \right.
\end{align*}
$\Psi \Xi \Psi^T$ is the Spectral decomposition of $(\widehat{A}^T\widehat{S}^T\widehat{S}\widehat{A})^{-1}$ and $\mu_i$ is the $i^{th}$ eigenvalue \cite{ref:mihs}. The characteristic polynomials of each block is
\begin{equation}
    u^2 - (1+\beta - \alpha \mu_i)u + \beta = 0, \quad \forall i \in [r].\label{eq:characteristic}
\end{equation}
If the following condition holds
\begin{equation}
    \beta \geq (1 - \sqrt{\alpha \mu_i})^2, \quad \forall i \in [r] \label{condition:dynamic},
\end{equation}
then both of the roots are imaginary and both have a magnitude $\sqrt{\beta}$ for all $\mu_i$'s. In this case, all linear dynamical systems driven by the above characteristic polynomial will be in the under-damped regime and the contraction rate of the transformation $T$, through all directions, not just one of them, will be exactly $\sqrt{\beta}$. If the condition in \eqref{condition:dynamic} is not satisfied for a $\mu_i$ with $i\in [r]$, then the linear dynamical system corresponding to $\mu_i$ will be in the over-damped regime and the contraction rate in the direction through the eigenvector corresponding to this over-damped system will be smaller compared to the others. As a result, the overall algorithm will be slowed down (see \cite{ref:candes_adaptive} for details). If the condition in \eqref{condition:1} of Theorem \ref{lemma:mihs_conv} holds,
\begin{equation*}
    \lscost{\widehat{A}^T\widehat{S}^T\widehat{S}\widehat{A} - I_r}{2} = \lscost{U_1^TS^TSU_1 +U_2^TU_2- I_r}{2}=\lscost{U_1^TS^TSU_1 - U_1^TU_1}{2}\leq \epsilon,
\end{equation*}
then, we have the following bounds:
\begin{equation*}
    \swrite{sup }{\lscost{v}{2} = 1}v^T\widehat{A}^T\widehat{S}^T\widehat{S}\widehat{A}v \leq 1+\epsilon \quad\mbox{   and   }\quad \swrite{inf }{\lscost{v}{2} = 1}v^T\widehat{A}^T\widehat{S}^T\widehat{S}\widehat{A}v \geq 1-\epsilon,
\end{equation*}
which are equivalent to:
\begin{equation*}
    \swrite{maximize}{ i \in [r]} \mu_i \leq \frac{1}{1-\epsilon} \quad \mbox{   and   }\quad \swrite{minimize}{ i \in [r]} \mu_i \geq \frac{1}{1+\epsilon}.
\end{equation*}
Consequently, the condition in \eqref{condition:dynamic} can be satisfied for all $\mu_i$'s by the following choice of $\beta$ that minimizes the convergence rate over step size $\alpha$
\begin{equation*}
    \sqrt{\beta^*} = \swrite{minimize}{\alpha}\!\left(\max\left\{1 - \frac{\sqrt{\alpha}}{\sqrt{1 +\epsilon}},\ \frac{\sqrt{\alpha}}{\sqrt{1 - \epsilon}} - 1 \right\}\right) =  \frac{\epsilon}{1 + \sqrt{1 -\epsilon^2}},
\end{equation*}
where the minimum is achieved at $\alpha^* = \frac{4(1 - \epsilon^2)}{(\sqrt{1+\epsilon}+\sqrt{1 - \epsilon})^2} = (1- \beta^*)\sqrt{1 - \epsilon^2}$ as claimed.
\paragraph{The proof for the Dual M-IHS}
The proof of the under-determined case is parallel to the over-determined one except for the following modifications. Let
\begin{equation}
    \widehat{A}^T = \begin{bmatrix} V\Sigma D \\ \sqrt{\lambda} UD\end{bmatrix}=\begin{bmatrix} U_1\\U_2\end{bmatrix}, \ \widehat{A}\widehat{A}^T = I_n \ \mbox{ and }\ \widehat{S}\widehat{A}^T = \begin{bmatrix} SV\Sigma D\\ \sqrt{\lambda}UD\end{bmatrix} = \begin{bmatrix} SU_1\\U_2\end{bmatrix},\label{eq:dual_notation}
\end{equation}
so that $U_1$ is the first $d$ rows of an orthogonal basis for $[A \ \ \sqrt{\lambda} I_n]$ as required by the theorem. Similar to the \ff{M-IHS} case, the Lyapunov analysis can be simplified by using the following formulation
\begin{equation*}
    w^* = \swrite{argmin }{w\in\reals{n}} = \frac{1}{2}\lscost{\widehat{A}^Tw}{2}^2 - \langle DU^Tb,\ w\rangle,
\end{equation*}
which is equivalent to the dual problem in \eqref{eq:intro_dual} due to the one-to-one mapping $\left\{\forall \nu(\lambda) \in\reals{n}\ |\ w^* = D^{-1}U^T\nu(\lambda) \right\}$. For this form, the equivalent of the \ff{Dual M-IHS} given in \eqref{eq:dual_mihs1} is
\begin{align*}
    \Delta w^{i} &= \swrite{argmin}{w} \lscost{\widehat{S}\widehat{A}^Tw}{2}^2 - 2\langle DU^Tb -  \widehat{A}\widehat{A}^Tw^i, \ w \rangle,\\
    w^{i+1} &= w^i + \alpha \Delta w^i + \beta(w^i - w^{i-1}).
\end{align*}
Therefore, we can analyze the following bipartite transformation to figure out the convergence properties of the \ff{Dual M-IHS}
\begin{equation*}
    \begin{bmatrix}
    w^{i+1} - w^* \\
    w^{i} - w^*
    \end{bmatrix} = \underbrace{\begin{bmatrix}
    (1+\beta)I_n - \alpha (\widehat{A}\widehat{S}^T\widehat{S}\widehat{A}^T)^{-1} & -\beta I_n\\
    I_n & \mathbf{0}
    \end{bmatrix}}_{T} \begin{bmatrix}
    w^{i} - w^* \\
    w^{i-1} - w^*
    \end{bmatrix}.
\end{equation*}
The rest of the proof can be completed straightforwardly by following the same analysis steps as in the proof for the \ff{M-IHS} case.

%% file: proofs/proof_2.tex
The following identities will be used on $U_1$ where $U,\Sigma$, and $D$ are defined in the proof of Theorem \ref{lemma:mihs_conv}:
\begin{equation*}
    \lscost{U_1}{F}^2 = \lscost{U\Sigma D}{F}^2 = \lscost{\Sigma(\Sigma^2 +\lambda I_r)^{-1/2}}{F}^2 = \sum\limits_{i=1}^{r}\frac{\sigma_i^2}{\sigma_i^2 + \lambda} = \sd(A),
\end{equation*}
and $\lscost{U_1}{2}^2 = \frac{\sigma_1^2}{\sigma_1^2 + \lambda} \approx 1$ for a properly chosen regularization parameter $\lambda$. If the sketch matrix $S$ is drawn from a randomized distribution $\D$ over matrices $\realss{m}{n}$, then by using the Approximate Matrix Property (AMM) which is given below, it will be proven that the condition in \eqref{condition:1} can be met with a desired level of probability.

As proven in \cite{ref:cnw15_optimal_stable_rank}, if a distribution $\D$ over $S\in\realss{m}{n}$ has the $(\epsilon, \delta, 2k,\ell)$-OSE moment property for some $\delta <1/2$ and $\ell\geq2$, then it has $(\epsilon, \delta, k)$-AMM Property for any $A,B$, i.e.,
\begin{equation}
    \prob{S\sim\D}{\lscost{A^TS^TSB - A^TB}{2} > \epsilon\sqrt{\left(\lscost{A}{2}^2+ \frac{\lscost{A}{F}^2}{k}\right)\left(\lscost{B}{2}^2 + \frac{\lscost{B}{F}^2}{k}\right)}} < \delta.\label{theorem:amm_ose}
\end{equation}
The definition of the OSE-moment property can be found in \cite{ref:cnw15_optimal_stable_rank}. As it will be detailed next, using the AMM property in \eqref{theorem:amm_ose}, the sketch sizes in the statement of Lemma \ref{corollary:sketch} can be found relative to the embedding size $k$ to satisfy the condition in \eqref{condition:1} of Theorem \ref{lemma:mihs_conv}.

For case $(i)$ of Lemma \ref{corollary:sketch}, Count Sketch with a single nonzero element in each column and size $m\geq 2/({\epsilon'}^2\delta)$ has $(\epsilon', \delta,2)$-JL moment property \cite{ref:tz12}. JL-Moment Property can be found in Definition 6.1 of \cite{ref:kn14_sparser_jl}. By Theorem 6.2 in \cite{ref:kn14_sparser_jl}:
\begin{equation*}
    \lscost{U_1S^TSU_1 - U_1^TU_1}{F} < 3\epsilon'\lscost{U_1}{F}^2 = 3\epsilon'\sd(A) \leq \epsilon
\end{equation*}
for $\epsilon' = \epsilon/(3\sd(A))$. So, condition in \eqref{condition:1} holds with probability at least $1-\delta$, if $m = O(\sd(A)^2/(\epsilon^2\delta))$. 

For case $(ii)$ of Lemma \ref{corollary:sketch}, combining Theorem 4.2 of \cite{ref:coh16_nearly_tight} and Remark 2 of \cite{ref:cnw15_optimal_stable_rank} implies that any sketch matrix drawn from an OSNAP \cite{ref:nn13_osnap} with the conditions given in case $(ii)$ of Lemma \ref{corollary:sketch} satisfies the $(\epsilon', \delta, k, \log(k/\delta))$-OSE moment property thus the $(\epsilon', \delta, k/2)$-AMM Property. Setting $A=B=U_1$  and $k = \sd(A)/2$ in \eqref{theorem:amm_ose} gives:
\begin{equation*}
    \lscost{U_1^TS^TSU_1 - U_1^TU_1}{2} \leq \epsilon'(\lscost{U_1}{2}^2 + 2) \leq 3 \epsilon' \leq \epsilon
\end{equation*}
with probability of at least $(1-\delta)$.
%%%%%%%%%%%%%%%%%%%%%%%%%%%%%
\begin{remark}
Based on the lower bounds established for any OSE in \cite{ref:nn14_lower}, the Conjecture 14 in \cite{ref:nn13_osnap} states that any OSNAP with 
$m=\Omega((k+\log(1/\delta))/\epsilon^2)$ and 
$s=\Omega(\log(k/\delta)/\epsilon)$ have the $(\epsilon,\delta,k,\ell)$-OSE moment property for $\ell = \Theta(\log(k/\delta))$, an even integer. If this conjecture is proved, then by the AMM property in \eqref{theorem:amm_ose}, the condition in \eqref{condition:1} can be satisfied with probability at least $(1-\delta)$ by using an OSNAP matrix with size $m=\Omega((\sd(A)+\log(1/\delta))/\epsilon^2)$ and sparsity $s=\Omega(\log(\sd(A)/\delta)/\epsilon)$.
\end{remark}

For case $(iii)$ of Lemma \ref{corollary:sketch}, by Theorem 9 of \cite{ref:cnw15_optimal_stable_rank}, SRHT with the sketch size given in case $(iii)$ has the $(\epsilon', \delta, 2\sd(A),$ $ \log(\sd(A)/\delta))$-OSE moment property and thus it provides $(\epsilon', \delta, \sd(A))$-AMM property. Again, setting $A=B=U_1$ and $k=\sd(A)$ in \eqref{theorem:amm_ose} produces the desired result.

For case $(iv)$ of Lemma \ref{corollary:sketch}, the Subgaussian matrices having entries with mean zero and variance $1/m$ satisfy the JL Lemma \cite{ref:jl_lemma} with optimal sketch size \cite{ref:kn14_sparser_jl}. Also, they have the $(\epsilon/2, \delta, \Theta(\log(1/\delta)))$-JL moment property \cite{ref:kmn_11_almost_optimal}. Thus by Lemma 4 of \cite{ref:cnw15_optimal_stable_rank} such matrices have $(\epsilon,\delta,k,\Theta(k+\log(1/\delta)))$-OSE moment property for $\delta < 9^{-k}$, which means $m=\Omega(k/\epsilon^2)$. Again, by setting $A=B=U_1$ and $k = \sd(A)$ in \eqref{theorem:amm_ose} produces the desired result.

%% file: proofs/proof_3.tex
Consider the regularized LS solution with parameter $\lambda$ and the Truncated SVD solution with parameter $\lceil\sd(A)\rceil$:
\begin{equation}
    x^* = \sum\limits_{i=1}^r \frac{\sigma_i^2}{\sigma_i^2 + \lambda}\frac{u_i^Tb}{\sigma_i}v_i \quad\mbox{ and }\quad x^\dagger = \sum\limits_{i=1}^{\lceil\sd(A)\rceil}\frac{u_i^Tb}{\sigma_i}v_i \label{eq:tsvdsol}
\end{equation}
where $u_i$'s and $v_i$'s are columns of $U$ and $V$ matrices in the SVD. The Tikhonov regularization with the closed form solution is preferred in practice to avoid the high computational cost of the SVD. The filtering coefficients of the Tikhonov regularization, $\frac{\sigma_i^2}{\sigma_i^2 + \lambda}$, become very close to the binary filtering coefficients of the TSVD, as the decay rate of the singular values of $A$ increases. In these cases, $x^*$ and $x^\dagger$ in \eqref{eq:tsvdsol} are very close to each other (Section 4 and 5 of \cite{ref:hansen_book1}). Thereby, the diagonal matrix $\Sigma D$ which is used in the proof of Lemma \ref{corollary:sketch} can be approximated by the diagonal matrix $\Pi$ where
\begin{equation*}
    \Pi_{ii} = \left\{
\begin{array}{rl}
1 & \text{if } i \leq \sd(A) \leq r\\
0 &  otherwise
\end{array} \right.,
\end{equation*}
which is equivalent to replacing the Tikhonov coefficients by the binary coefficients. Then, we have the following close approximation:
\begin{align*}
    \left(\widehat{A}^T\widehat{S}^T\widehat{S}\widehat{A}\right)^{-1} &= \left(D\Sigma U^TS^TSU\Sigma D + \lambda D^2\right)^{-1}\\ &\approx \left(\Pi (SU)^T(SU)\Pi + I_r - \Pi\right)^{-1} = \left[\begin{array}{c|c}
        \widebar{S}^T\widebar{S}& \textbf{0}\\
        \hline
        \textbf{0} & I_{(r - \sd(A))}
    \end{array}\right]^{-1},
\end{align*}
where $\widebar{S} = SU\Pi\in\realss{m}{\sd(A)}$ has the same distribution as $S$, since $U\Pi$ is an orthonormal transformation. By the MPL, the minimum and the maximum eigenvalues of this approximation converge to $\left(1 \pm \sqrt{\frac{\sd(A)}{m}}\right)^{-2}$ as $m\to \infty$ and while $\sd(A)/m$ remains constant \cite{ref:edelman}. The rest of the proof follows from the analysis given in the proof of Theorem \ref{lemma:mihs_conv}.

%% file: demos/fig_convergence.tex
\begin{figure}[h]
\subfigure[{Dense problem, SRHT sketch via DCT} \label{fig:convergence_1}]{\includegraphics[width=0.5\linewidth]{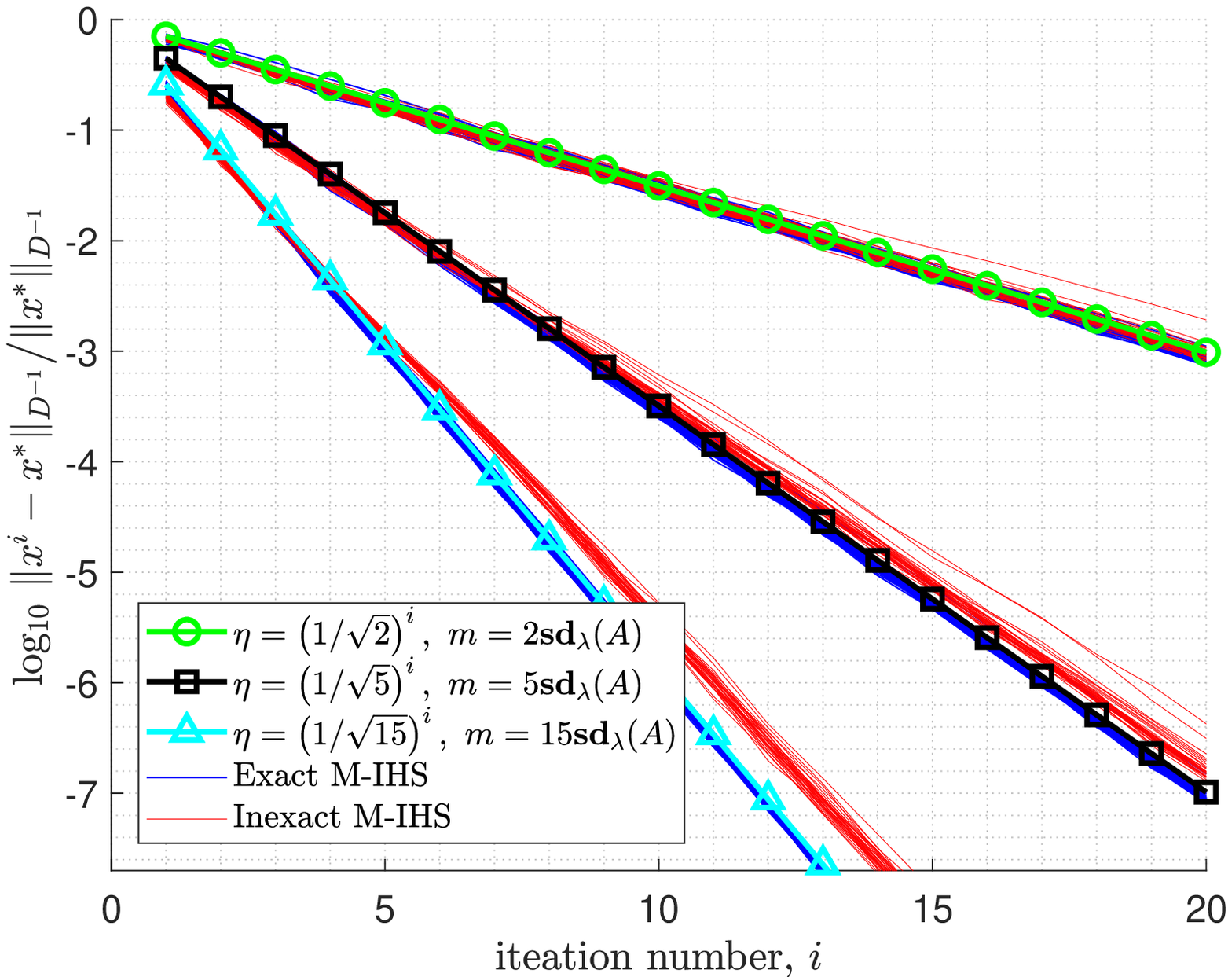}}%
\subfigure[{Sparse problem, Count sketch}\label{fig:convergence_2}]{\includegraphics[width=0.5\linewidth]{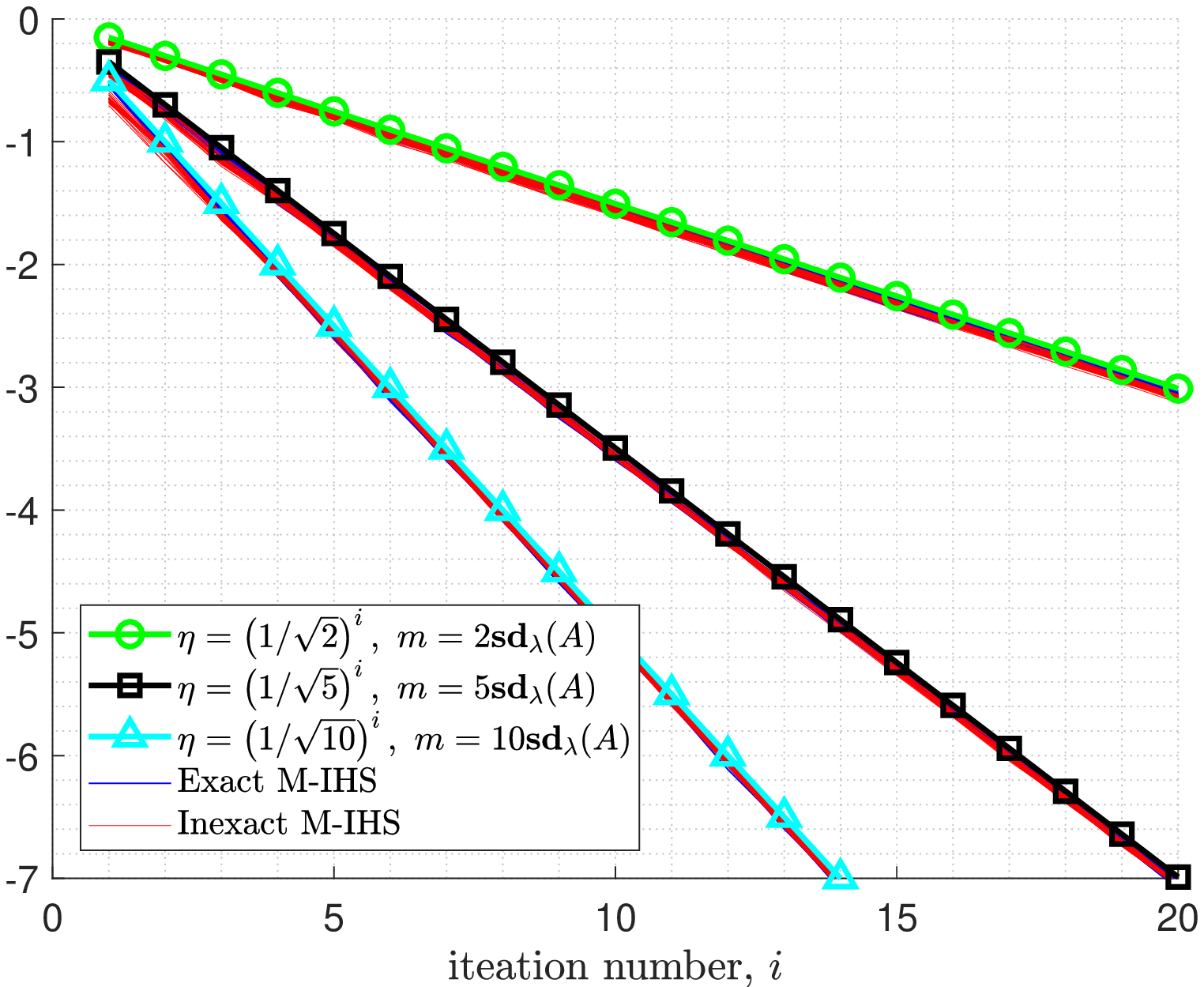}}
\caption{{Comparison of the theoretical rate given in Corollary \ref{corollary:empirical} and the empirical convergence rate. The lines with different markers show the theoretical convergence rate for different sketch sizes. Both the \textit{exact} and the \textit{inexact} (given in Algorithm \ref{algo:mihs}) versions of the \ff{M-IHS} were run $32$ times and the result of each run is plotted as a separate line. Except for a small degradation in the dense case, setting the forcing term to a small constant such as $\epsilon_{sub} = 0.1$ is sufficient for the inexact scheme to achieve the same rate as the exact version in these experiments.}}
\label{fig:convergence}
\end{figure}

%% file: demos/algo_mihs.tex
\begin{algorithm}[!htbp]
\setstretch{1.2}
\caption{\ff{M-IHS} (for $n\geq d$)}
\label{algo:mihs}
\begin{algorithmic}[1]
\STATE \textit{Input}: $A, \ b,\ m,\ \lambda,\ N,\ x^1,\ \sd(A),\ \epsilon_{sub}$\hfill $\gray{complexity}$
\STATE \hspace{1pt}$S\!A = $ \ff{RP\_fun}$(A,m)$\hfill$\gray{C(m,n,d)}$
\STATE $\phantom{A}\beta = \sd(A)/m$\hfill$\gray{O(1)}$  
\STATE$\phantom{A}\alpha = (1- \beta)^2$\hfill$\gray{O(1)}$
\FOR{$i=1:N$}
\STATE$\hspace{8pt}g^i = A^T(b - Ax^i) - \lambda x^i$ \hfill$\gray{4nd + 3d}$
\STATE$\hspace{2pt}\Delta x^{i} = $ \ff{AAb\_Solver}$(S\!A, \ g^i,\  \lambda,\  \epsilon_{sub})$ \hfill $\gray{O(md)}$ 
\STATE${x^{i+1}} = x^i + \alpha \Delta x^i + \beta(x^i - x^{i-1})$\hspace{275pt}$\gray{5d }$ 

\ENDFOR
\end{algorithmic}
\end{algorithm}%
\begin{algorithm}[!htbp]
\setstretch{1.2}
\caption{\ff{Dual M-IHS} (for $n\leq d$)}
\label{algo:dual_mihs}
\begin{algorithmic}[1]
\STATE \textit{Input}: $A, \ b,\ m,\ \lambda,\ N,\ \sd(A),\ \epsilon_{sub}$\hfill$\gray{complexity}$
\STATE \hspace{1pt}$S\!A^T = $ \ff{RP\_fun}$(A^T,m)$\hfill$\gray{C(m,n,d)}$
\STATE $\phantom{A^T}\beta = \sd(A)/m$\hfill$\gray{O(1)}$  
\STATE$\phantom{A^T}\alpha = (1- \beta)^2$\hfill$\gray{O(1)}$
\STATE\hspace{10pt}$\nu^1 = 0$\hfill$\gray{O(1)}$
\FOR{$i=1:N$}
\STATE$\hspace{8pt}g^i = b - AA^T\nu^i -\lambda\nu^i$\hfill$\gray{4nd+3n}$
\STATE$\hspace{2pt}\Delta \nu^{i} = $ \ff{AAb\_Solver}$(S\!A^T, \ g^i,\  \lambda,\  \epsilon_{sub})$\hfill$\gray{O(mn)}$
\STATE${\nu^{i+1}} = \nu^i + \alpha \Delta \nu^i + \beta(\nu^i - \nu^{i-1})$ \hspace{275pt}$\gray{5n}$
\ENDFOR
\STATE${x}^{N+1} = A^T{\nu}^{N+1}$\hfill$\gray{2nd}$
\end{algorithmic}
\end{algorithm}%

%% file: demos/algo_pd.tex
\begin{algorithm}[!htbp]
\setstretch{1.2}
\caption{ \ff{Primal Dual M-IHS} (for $n\leq d$)}
\label{algo:pd_mihs_1}
\begin{algorithmic}[1]
\STATE \textit{Input}: $A, \ b,\ m_1,\ m_2, \ \lambda, \ N, \ M,\ \sd(A),\ \epsilon_{sub}$\hfill$\gray{complexity}$
\STATE $\hspace{9pt}SA^T = $ \ff{RP\_fun}$(A^T,m_1)$\hfill$\gray{C(m_1,n,d)}$
\STATE $WAS^T = $ \ff{RP\_fun}$(SA^T,m_2)$\hfill$\gray{C(m_1,m_2, n)}$
\STATE $\hspace{15pt}\beta_\ell = \sd(A)/m_\ell,\hspace{15pt} \ell=1,2$\hfill$\gray{O(1)}$  
\STATE$\hspace{15pt}\alpha_\ell = (1- \beta_\ell)^2, \hspace{22pt} \ell = 1,2$\hfill$\gray{O(1)}$
\STATE\hspace{16pt}$\nu^{1} = 0, \ z^{1,1} = 0$\hfill$\gray{O(1)}$
\FOR{$i=1:N$}
\STATE${b}^i = b - AA^T\nu^i - \lambda\nu^i$\hfill$\gray{4nd+3n}$
\FOR{$j=1:M$}
\STATE\hspace{8pt}$g^{i,j} = SA^T({b}^i - AS^T z^{i,j}) - \lambda z^{i,j}$\hfill$\gray{4nm_1+3m_1}$
\STATE$\hspace{2pt}\Delta z^{i,j} =$ \ff{AAb\_Solver}$(WAS^T,  g^{i,j},\  \lambda,\ \epsilon_{sub})$\hfill$\gray{O(m_1m_2)}$
\STATE${z^{i,j+1}} = z^{i,j} + \alpha_2 \Delta z^{i,j} + \beta_2(z^{i,j} - z^{i,j-1})$\hspace{230pt}$\gray{5m_1}$
\ENDFOR
\STATE\hspace{2pt}$\Delta\nu^i = ({b}^i - AS^Tz^{i,M+1})/\lambda, \quad z^{i+1,1} = z^{i,M+1}$\hfill$\gray{2nm_1+2n}$ 
\STATE$\nu^{i+1} = \nu^i + \alpha_1 \Delta\nu^i +\beta_1(\nu^i - \nu^{i-1})$\hspace{271pt}$\gray{5n}$
\ENDFOR
\STATE$x^{N+1} = A^T\nu^{N+1}$\hfill$\gray{2nd}$
\end{algorithmic}
\end{algorithm}%
\begin{algorithm}[!htbp]
\setstretch{1.2}
\caption{\ff{Primal Dual M-IHS} (for $n\geq d$)}
\label{algo:pd_mihs_2}
\begin{algorithmic}[1]
\STATE \textbf{Input}: $A, \ b,\ m_1,\ m_2, \ N, M, \ \lambda,\ x^1,\ \sd(A),\ \epsilon_{sub}$\hfill$\gray{complexity}$
\STATE $\hspace{19pt}SA = $ \ff{RP\_fun}$(A,m_1)$\hfill$\gray{C(m_1,n,d)}$
\STATE $WA^TS^T = $ \ff{RP\_fun}$((SA)^T,m_2)$\hfill$\gray{C(m_1,m_2,d)}$
\STATE $\hspace{21pt}\beta_\ell = \sd(A)/m_\ell,\hspace{15pt} \ell=1,2$\hfill$\gray{O(1)}$  
\STATE$\hspace{21pt}\alpha_\ell = (1- \beta_\ell)^2, \hspace{22pt} \ell = 1,2$\hfill$\gray{O(1)}$
\STATE$\hspace{17pt}z^{1,1} = z^{1,0} = 0$\hfill$\gray{O(1)}$
\FOR{$i=1:N$}
\STATE${b}^i = A^T(b - Ax^i) - \lambda x^i$\hfill$\gray{4nd+3d}$
\FOR{$j=1:M$}
\STATE $\hspace{7pt}{g}^{i,j} = SA({b}^i - A^TS^Tz^{i,j}) - \lambda z^{i,j}$\hfill$\gray{4dm_1+3m_1}$
\STATE$\hspace{2pt}\Delta z^{i,j} =$ \ff{AAb\_Solver}$(WA^TS^T,  g^{i,j},\  \lambda,\ \epsilon_{sub})$\hfill$\gray{O(m_1m_2)}$
\STATE${z^{i,j+1}} = z^{i,j} + \alpha_2 \Delta z^{i,j} + \beta_2(z^{i,j} - z^{i,j-1})$\hspace{230pt}$\gray{5m_1}$
\ENDFOR
\STATE$\hspace{2pt}\Delta x^i = ({b}^i- A^TS^Tz^{i,M+1})/\lambda, \quad z^{i+1,1} = z^{i,M+1}$\hfill$\gray{2dm_1+2d}$ 
\STATE$x^{i+1} = x^i + \alpha_1 \Delta x^i +\beta_1(x^i - x^{i-1})$\hspace{269pt}$\gray{5d}$
\ENDFOR
\end{algorithmic}
\end{algorithm}

%% file: sections/trace_estimator.tex
The statistical dimension $\sd(A)$ in Algorithms \ref{algo:mihs}, \ref{algo:dual_mihs}, \ref{algo:pd_mihs_1} and \ref{algo:pd_mihs_2} can be estimated by using a Hutchinson-like randomized trace estimator \cite{ref:avron_hutchinson}. Alternatively, $\sd(A)$ can be estimated by using the algorithm proposed in \cite{ref:acw16_sharper} within a constant factor in $\mathbf{nnz}(A)$ time with a constant probability, if $\sd(A) \leq \xi$ where:
\begin{equation*}
    \xi = \min\{n,d,\lfloor(n+d)^{1/3}/\mbox{poly}(\log(n+d)) \rfloor\}.
\end{equation*}
However, due to the third order root and the division by typically higher than a sixth order polynomial, $\xi$ becomes very small and the proposed algorithm in \cite{ref:acw16_sharper} can only be used when the singular values of $A$ decay severely/exponentially. %
\input{demos/algo_trace_est}%
Therefore, we preferred to use the heuristic trace estimator in Algorithm \ref{algo:trace}, where the input matrix $SA$ can be replaced with $SA^T$ or even with $WA^TS^T$ and $WAS^T$ according to the requirements of the algorithm used. Any estimator in \cite{ref:avron_hutchinson} can be substituted for the Hutchinson Estimator and the number of samples $T$ can be chosen accordingly. In the conducted experiments with various singular value profiles, small samples sizes such as $2$ or $3$ and $\epsilon_{tr} = 0.5$ was sufficient to obtain satisfactory estimates for $\sd(A)$ used in Corollary \ref{corollary:empirical}. Note that, as long as $\sd(A)$ is overestimated, the convergence rates of the proposed algorithms will be strictly controlled by $\beta$ as in Corollary \ref{corollary:empirical}.

%% file: demos/algo_trace_est.tex
\begin{algorithm}[!htbp]
\caption{\ff{Inexact Hutchinson Trace Estimator}}
\label{algo:trace}
\begin{algorithmic}[1]
\STATE \textbf{Input:} $SA\in\realss{m}{d},\ \lambda,\ T, \ \epsilon_{tr}$\hfill$\gray{complexity}$
\STATE $v^\ell = \{-1, +1\}^d$, \quad $\ell = 1,\ldots,T$\hfill$\gray{O(Td)}$
\STATE $\tau = 0$\hfill$\gray{O(1)}$
\FOR{i = 1:$T$}
\STATE $z^i = $ \ff{AAb\_Solver}$(SA, v^i, \lambda, \epsilon_{tr})$\hfill$\gray{O(md)}$
\STATE \hspace{2pt}$\tau = \tau + \lambda \langle v^i,\ z^i\rangle$\hspace{328pt}$\gray{2d}$
\ENDFOR
\STATE \textbf{Output}: $ \widehat{\sd} = d - \tau /T$\hfill$\gray{O(1)}$
\end{algorithmic}
\end{algorithm}

%% file: sections/complexity2.tex
%\documentclass[main.tex]{subfiles}
%\begin{document}
The iterations of both the exact and inexact \ff{M-IHS} and \ff{Dual M-IHS} consist of 4 stages with the computational complexities given in Table \ref{table:1}.
\begin{table}[!htpb]\small
\caption{Computational complexity of each stage in the M-IHS techniques}
\label{table:1}
\begin{equation*}
    \begin{array}{lcc}
        \textbf{Stage}&\textbf{Exact schemes} & \textbf{Inexact schemes}\\\hline 
         \textbf{generation of } SA \textbf{ or } SA^T & C(n,d,\sd(A)) & C(n,d,\sd(A)) \\
         \textbf{QR }(R-\textbf{factor only)} & O(r\sd(A)^2) &  N.A. \\
         \sd(A) \textbf{ estimation} & O(T\sd(A)^2)& O(T\sqrt{\kappa(\lambda)}\log(\epsilon_{tr}^{-1})\sd(A)r)\\
         1 \textbf{ iteration}& O(nd+\sd(A)^2) &  O(nd + \sqrt{\kappa(\lambda)}\log(\epsilon_{sub}^{-1})\sd(A)r)\\
    \end{array}
    \end{equation*}
\end{table}\normalsize
There, $T$ is the number of samples used in Hutchinson-like estimators, $\kappa(\lambda)$ is the condition number of $A^TA+\lambda I$, $r  =\min(n,d)$, and $C(\cdot)$ is the complexity of generating the sketched matrix which is noted in Lemma \ref{corollary:sketch}. For the proposed techniques, sketch size $m$ can be chosen proportional to the statistical dimension $\sd(A)$ which is always smaller than $r$. We assumed that the sub-problems in \eqref{eq:mihs1} and \eqref{eq:dual_mihs1} are solved by using the QR decomposition for the exact schemes. The tolerance parameters $\epsilon_{tr}$ and $\epsilon_{sub}$, that are used to terminate the sub-solvers, are typically chosen around $0.5$ or $0.1$ as noted in Sections \ref{sec:efficent_sub} and \ref{sec:sd}. The proposed techniques provide two major computational advantages over the current randomized preconditioning solvers: the first is the capability to use the sketch sizes that are smaller than $r$, and the second is the ability of avoiding the complexity of the QR step. In the inexact \ff{M-IHS} variants, the third order complexity $O(\min(mr^2, rm^2))$ of matrix decomposition or inversion are avoided by solving sub-problems in each iteration via Krylov Subspace-based solvers. Although these sub-problems incur a complexity of $O(\sqrt{\kappa(\lambda)}mr)$, the overall complexity can be reduced significantly. For the applications where $m$ grows larger, this saving becomes critical as shown in Section \ref{sec:varying}. The memory space required by all the above techniques is $O(\sd(A)r)$. In Table \ref{table:1}, it is assumed that $A$ is a dense matrix. For those cases where $A$ is sparse, further computational savings can be achieved for the proposed techniques. In the following Table \ref{table:2}, the overall computational complexities of algorithms to obtain an $\eta-$optimal solution approximation are given for both dense and sparse $A$'s.
\begin{table}[!htpb]
\caption{Overall computational complexity against baselines}
\label{table:2}
\begin{equation*}\small
    \begin{array}{lcc}
        &\textbf{Dense } A& \textbf{Sparse } A\\\hline 
        \textbf{Inexact M-IHS} & O\left(nd\log(\sd(A)) + \log\left(\eta^{-1}\right)\left(nd + \sqrt{\kappa(\lambda)}\sd(A)r\right)\right)&  O\left(\nnz(A) + \log(\eta^{-1})\left(\nnz(A) +  \sqrt{\kappa(\lambda)} \sd(A)r\right)\right)\\
        \textbf{Exact M-IHS} & O\left(nd\log(\sd(A)) + r\sd(A)^2 + \log\left(\eta^{-1}\right)nd\right) & O\big(\nnz(A) + r\sd(A)^2 + \log\left(\eta^{-1}\right)\nnz(A)\big) \\
        \textbf{Blendenpik} & O\left(nd\log(r) + r^3 + \log\left(\eta^{-1}\right)nd\right) &O\left(nd\log(r) + r^3 + \log\left(\eta^{-1}\right)\nnz(A)\right)\\
        \textbf{LSQR/CG} & O\left(\sqrt{\kappa(\lambda)}\log(\eta^{-1})nd\right) & O\left(\sqrt{\kappa(\lambda)}\log\left(\eta^{-1}\right)\nnz(A)\right)\\
        \textbf{QR} & O(ndr) &  O(ndr)^{(2)}\\
    \end{array}
\end{equation*}\normalsize
\end{table}
In Table \ref{table:2}, we assumed that the SRHT sketch matrices are used for a dense $A$ while the sparse subspace embeddings with $s$ nonzero entry in each column are used for a sparse $A$ with $\nnz(A)$ number of nonzero entries{\let\thefootnote\relax\footnote{$^2$Note that for sparse $A$, the QR algorithm that is based on Givens rotation does not provide significant complexity reduction in theory since sequentially applied Givens rotation causes fill-in in the data matrix.}}. Depending on the type of choice of the sketching used, the complexity of the proposed techniques vary significantly. For dense coefficient matrices while the SRHT matrices has lower run time in sequential environments, Gaussian matrices would be more efficient in parallel computing. If the coefficient matrix is sparse, then the data oblivious sketching types such as OSNAP or CountSketch matrices would be effective choices with run time of $O(s\nnz(A))$ where $s$ is proportional to $\log(\sd(A))$ as noted in Lemma \ref{corollary:sketch}. The proposed techniques can be still used even if the coefficient matrix is an operator, in this case Gaussian or sparse embeddings can be utilized. If the coefficient matrix is sparse or an operator that allows fast matrix-vector computations, then both the exact and inexact schemes are automatically sped up due to the saving in the gradient computation. For instance, in the sparse case, complexity of the gradient computation is reduced from $O(nd)$ to $O(\text{nnz}(A))$ as shown in Table \ref{table:2}.

In a similar manner, the complexity of each stage in the \ff{Primal Dual M-IHS} variants is given in Table \ref{table:3}.
\begin{table}[!htpb]
\caption{Computational complexity of each stage in the Primal Dual M-IHS techniques}
\label{table:3}
\begin{equation*}\small
    \begin{array}{lcc}
        \textbf{Stage}& \textbf{Exact schemes} & \textbf{Inexact schemes}\\\hline 
         \textbf{generation of } W\!AS^T \textbf{ or } W\!A^T\!S^T& C(n,d,\sd(A))+C(r,\sd(A),\sd(A)) &  C(n,d,\sd(A))+C(r,\sd(A),\sd(A)) \\
         \textbf{QR or SVD} & O(\sd(A)^3) &  N.A. \\
         \sd(A) \textbf{ estimation} & O(T\sd(A)^2)& O\left(T\sqrt{\kappa(\lambda)}\log(\epsilon_{tr}^{-1})\sd(A)^2\right)\\
         1 \textbf{ outer iteration}& O\left(nd + M(r\sd(A)+\sd(A)^2)\right) &  O\left(nd + M\sd(A)\left(r + \sqrt{\kappa(\lambda)}\log(\epsilon_{sub}^{-1})\sd(A) \right)\right)\\
    \end{array}
\end{equation*}\normalsize
\end{table}
Here, $M$ denote the number of inner iterations, and both sketch size $m_1$ and $m_2$ can be chosen proportional to the statistical dimension $\sd(A)$, e.g., $m_1=m_2=2\sd(A)$ as demonstrated in Figure \ref{fig:square_reg}. Unless $M\sd(A)\ll r$, the \ff{Primal Dual M-IHS} does not provide significant saving over the \ff{M-IHS} or the \ff{Dual M-IHS}. However, when $n$ and $d$ scale similar and the ratio $\sd(A)/r$ is very small, if the decomposition of the sketched matrix is required for parameter estimation purpose as discussed earlier (see Chapter 4 of \cite{ref:iko_ms}), then due to the decomposition of $m_2\times m_1$-dimensional doubly sketched matrix, the \ff{Primal Dual M-IHS} variants require far fewer operations then any exact schemes which need to compute the decomposition of $r\times m_1$ dimensional sketched matrix. Such conditions are prevalent, for example, in image de-blurring or seismic travel-time tomography problems \cite{ref:ir_tool}. The memory space required by \ff{Primal Dual M-IHS} techniques is $O(\sd(A)r+\sd(A)^2)$.
%\end{document}

%% file: sections/aab_solver.tex
The linear sub-problems in the form of $(A^TA+\lambda I)x=b$, whose solutions are required by all four of the proposed \ff{M-IHS} variants, can be approximately solved by using the \textit{bidiag2} procedure described in \cite{ref:lsqr}, which produces an upper bidiagonal matrix as:
\begin{equation*}
    P_{k}^TAV_k = R_k =\begin{bmatrix}
    \rho_1 & \theta_2 \\
           & \ddots & \ddots\\
           & & \rho_{k-1} & \theta_k\\
           & & & \rho_k\end{bmatrix}\in \realss{k}{k},
\end{equation*}
where $P_k \in\realss{n}{k}$, $V_k\in\realss{d}{k}$ and $P_k^TP_k=V_k^TV_k = I_k$. The upper bidiagonal decomposition $R_k$ is computed by using the Lanczos-like three term recurrence:
\begin{equation*}
    \begin{split}
        AV_k &= P_{k}R_k  \\
        A^TP_{k} &= V_kR_k^T + \theta_{k+1}v^{k+1}e_{k}^T  
    \end{split} \Longrightarrow
    \begin{split}
         Av^1 &= \rho_1 p^1,\\
        A^Tp^{j} &= \rho_jv^j + \theta_{j+1} v^{j+1} \hspace{4pt} \quad j\leq k,\\
        Av^j &=\theta_{j}p^{j-1}+\rho_{j}p^{j}, \qquad j\leq k,
    \end{split}  
\end{equation*}
where $\theta_j$'s and $\rho_j$'s are chosen so that $\|v^j\|_2=\|p^j\|_2=1$, respectively. Note that $P_k$ and $V_k$ are not needed to be orthogonal in \ff{AAb\_Solver}, therefore we do not need any reorthogonalization steps. Unlike the LSQR, we choose $\theta_1v^1 = b$ with $\theta_1 = \|b\|_2$ so that the columns of the matrix $V_k$ constitute an orthonormal basis for the $k$-th order Krylov Subspace:
\begin{equation*}
    \spans\{v^1,\ldots,v^k\} = \K_k(A^TA, \ b) = \K_k(A^TA+\mu I_d,\ b), \ \forall \mu\in\reals{}_+. 
\end{equation*}
Since the Krylov Subspace is invariant under a constant shift, regularization does not affect this property. In the $k$-th iteration of the proposed \ff{AAb\_Solver}, let the solution estimate of the linear system be $x^k = V_ky^k$ for some vector $y^k\in\reals{k}$, i.e., $x^k \in\K_k(A^TA, \ b)$, then we have 
$
    (A^TA + \lambda I_d)V_ky^k = b
$ 
which implies
\begin{equation*}
    \widebar{R}_ky^k = \widebar{R}_k^{-T}V_k^Tb \stackrel{(a)}{=} \theta_1 \widebar{R}_k^{-T}e_1,
\end{equation*}
where $(a)$ is due to the choice of $v^1$ and $\widebar{R}_k$ is obtained by applying a sequence of Givens rotation on $[R_k^T \ \sqrt{\lambda}I_k]^T$ in order to eliminate the sub-diagonal elements due to the regularization \cite{ref:elden_shift}. One instance of this elimination procedure is
\begin{equation*}
    \left[\begin{array}{cc}
        \widebar{\rho}_k    &   \theta_{k+1}\\
         0              &   \rho_{k+1} \\
         0              &   0 \\
         0              &   \sqrt{\lambda}
    \end{array}\right] \rightarrow \left[\begin{array}{cc}
        \widebar{\rho}_k    &   c_k\theta_{k+1}\\
         0              &   \rho_{k+1} \\
         0              &   0 \\
         0              &   \widebar{\lambda}_{k+1}
    \end{array}\right] \rightarrow \left[\begin{array}{cc}
        \widebar{\rho}_k    &   \widebar{\theta}_{k+1}\\
         0              &   \widebar{\rho}_{k+1} \\
         0              &   0 \\
         0              &   0
    \end{array}\right] \xrightarrow[\mbox{iteration}]{\mbox{next}} 
    \left[\begin{array}{cc}
        \widebar{\rho}_{k+1}    &   \theta_{k+2}\\
         0              &   \rho_{k+2} \\
         0              &   0 \\
         0              &   \sqrt{\lambda}
    \end{array}\right],
\end{equation*}
where $ c_k = \rho_k/\widebar{\rho}_k$, $s_k = \widebar{\lambda}_k/\widebar{\rho}_k$, $\widebar{\theta}_{k+1} = c_k\theta_{k+1}$, $\widebar{\lambda}_{k+1}^2 = \lambda + (s_k{\theta_{k+1})}^2$ and $\widebar{\rho}_{k+1}=\sqrt{\rho_{k+1}^2+\widebar{\lambda}_{k+1}^2}$. Since $\widebar{R}_k$ is an upper bidiagonal matrix, the inverse always exists and $f^k:=\widebar{R}_k^{-T}e_1$ can be computed analytically as:
\begin{equation}
    \phi_1 = \frac{\theta_1}{\widebar{\rho}_1} \quad \mbox{and}\quad \phi_k = -\phi_{k-1}\frac{\widebar{\theta}_k}{\widebar{\rho}_k} \text{ where } f^k = [\phi_1,\ldots,\phi_k]^T. \label{eq:upper_difference}
\end{equation}
Furthermore, the solution at the $k$-th iteration, $x^k = V_k\widebar{R}_k^{-1}f^k$, can be obtain without computing any inversions by using the forward substitution. Define $D_k = V_k\widebar{R}_k^{-1}$:
\begin{equation*}
\left.  \begin{array}{rl}
    [D_{k-1}, \ d^k]\begin{bmatrix}
    \widebar{R}_{k-1} & e_{k-1}\widebar{\theta}_{k}\\
    0 & \widebar{\rho}_k\end{bmatrix} =&\!\!\! [V_{k-1},\ v^k]\\
    D_{k-1}\widebar{R}_{k-1} =&\!\!\! V_{k-1} \\
    \widebar{\theta}_{k}d^{k-1} + \widebar{\rho}_kd^k =&\!\!\! v^k
\end{array}\right\}  \begin{array}{rl}
    d^k =&\!\!\! (v^k - \widebar{\theta}_kd^{k-1})/{\widebar{\rho}_k} \\
    x^{k} =&\!\!\! x^{k-1} + \phi_{k}d^k,
\end{array}
\end{equation*}
and the relative residual error that will be used as a stopping criterion can be found as:
\begin{align*}
    \left\|A^TAx^k + \lambda x^{k} - b\right\|_2^2 &= \left\|A^TAV_ky^k + \lambda V_ky^k - b\right\|_2^2 = \left\|A^TP_k\widebar{R}_ky^k - b\right\|_2^2 = \left\|\left(V_k\widebar{R}_k^T+\theta_{k+1}v^{k+1}e_k^T\right)\widebar{R}_ky^{k} - b\right\|_2^2 \\
    &\stackrel{(i)}{=}{\left\|\widebar{R}_k^T\widebar{R}_ky^{k} - V_k^Tb\right\|_2^2} + \left\|\theta_{k+1}v^{k+1}e_k^T\widebar{R}_ky^{k} - \left(I-V_kV_k^T\right)b\right\|_2^2 = \left|\phi_k\widebar{\theta}_{k+1}\right| = \left|\phi_{k+1}\widebar{\rho}_{k+1}\right|.
\end{align*}
The first norm in $(i)$ is zero since the linear system is always consistent. The second term in the second norm is also zero, since $b\in\mbox{span}(V_k)$ by the initial choice of $\theta_1v^1 = b$. By definition, $f^k = \widebar{R}_ky^k$ gives the final results. The overall algorithm is given in Algorithm \ref{algo:aab_solver}. \input{demos/algo_aab_solver.tex}The \ff{AAb\_Solver} is also a Krylov Subspace method, therefore, it finds the solution in at most $\min(n,d,m)$ iterations in the exact arithmetic, but far fewer number of iterations is sufficient for our purpose.

Efficient solutions for linear systems in the form of $(A + \lambda I)x = b$ for a symmetric matrix or $(A^TA + \lambda I)x = b$ for a rectangular matrix have been well studied subject. In the first case, Lanczos tridiagonalization algorithm can be used for deriving a stable solver \cite{ref:symmetric_systems}. In the second case, which is our main concern, if the lower bidiagonalization processes (\textit{bidiag1} in \cite{ref:lsqr}) is used such as in \cite{ref:gazzola}, then a tridiagonal system in the form of $B_k^TB_ky^k = \theta_1e_1$ must be solved where $B_k\in\realss{k+1}{k}$ is a lower bidiagonal matrix. This system can be solved by first eliminating the lower diagonal elements in the tridiagonal matrix $B_k^TB_k$ and then by using forward substitution. However, the condition number of $B_k^TB_k$ is the square of the condition number of $B_k$ and thus increases the instability of the operations in the inexact arithmetic. Therefore, in the proposed \ff{AAb\_Solver}, we use upper bidiagonalization process to solve a tridiagonal system in the form of $R_k^TR_ky^k = \theta_1e_1$. The major advantage of this form over the one obtained by lower bidiagonal matrix $B_k$ is that $R_k^{-T}e_1$ can be calculated analytically as in \eqref{eq:upper_difference}. Then the solution $y^k$ can be obtained via forward substitution. In this way, we avoid both squaring the condition number and the elimination process of the lower diagonal entries. As a result, we obtain a solver with better stability properties and with slightly lower computational requirements.

%% file: demos/algo_aab_solver.tex
\begin{algorithm}
\caption{\ff{AAb\_Solver} (for problems in the form of $(A^TA+\lambda I)x = b$)}
\label{algo:aab_solver}
\begin{algorithmic}[1]
\STATE Input:  $A \in \realss{m}{n}, b, \lambda, \epsilon$ \hfill$\vartriangleright$ choose $\rho$ and $\theta$ to make $\|p\|_2=\|v\|_2=1$\hfill$\gray{complexity}$
\STATE $\theta_1v = b$\hfill$\gray{3n}$
\STATE $\rho p = Av$\hfill$\gray{2mn + 3m}$
\STATE $ $
\STATE $\Bar{\rho} = \sqrt{\rho^2 + \lambda}, \quad c = \rho/\Bar{\rho}, \quad s = \sqrt{\lambda/\Bar{\rho}}, \quad \phi = \theta_1/\Bar{\rho}$, $t = \infty$\hfill$\gray{O(1)}$
\STATE$d = v/\Bar{\rho}$\hfill$\gray{n}$
\STATE$x = \phi d$\hfill$\gray{n}$
\WHILE{$t \geq \epsilon$}
\STATE$\theta v := A^Tp - \rho v$\hfill$\gray{2mn + 5n}$
\STATE$\rho p := Av - \theta p$ \hfill$\gray{2mn + 5m}$
\STATE$ $
\STATE$\hspace{1pt}\Bar{\lambda}^2 := \lambda + (s\theta)^2, \quad \Bar{\theta} = c\theta$\hfill$\gray{O(1)}$
\STATE$\phantom{\theta}\Bar{\rho} := \sqrt{\rho^2 + \Bar{\lambda}^2}, \quad c = \rho/\Bar{\rho}, \quad s = \Bar{\lambda}/\Bar{\rho}$\hfill$\gray{O(1)}$
\STATE$ $
\STATE$\phantom{\phi}d := (v - \Bar{\theta} d )/\Bar{\rho}$\hfill$\gray{3m}$
\STATE$\phantom{\phi}\phi := -\phi\Bar{\theta}/\Bar{\rho}$\hfill$\gray{O(1)}$
\STATE$\phantom{\phi}x := x +\phi d$\hfill$\gray{2n}$
\STATE$\phantom{\phi:}t = |\phi\Bar{\rho}|/\theta_1$\hspace{337pt}$\gray{O(1)}$
\ENDWHILE
\end{algorithmic}
\end{algorithm}

%% file: sections/numeric.tex
We compare the operation counts required by the algorithms to obtain a certain level of accuracy in the solution approximation metric. For a fair comparison, we have implemented all the proposed algorithms in this manuscript as well as those that are used for the comparisons in MATLAB which can be found in the provided in \url{https://github.com/ibrahimkurban/M-IHS}. 

\subsection{Experimental setting} \label{sec:ex_setup}
The coefficient matrix $A\in\realss{n}{d}$ was generated for various sizes as follows: we first sampled the entries of $A$ from the distribution $\mathcal{N}(1_d, \Gamma)$ where $\Gamma_{ij} = 5 \cdot 0.9^{|i-j|}$ so that the columns are highly correlated with each other. Then by using the SVD, we replaced the singular values with \textit{philips} profile provided in RegTool \cite{ref:hansen_matlab}. We scaled the singular values to set the condition number $\kappa(A)$ to $10^8$ and we used the same input signal provided by RegTool. In this way, we have obtained a challenging setup for any first order iterative solvers to compare their performances. In all the experiments, the same setup has been used unless indicated. We counted the number of operations according to Hunger's report \cite{ref:tum}. All the reported results have been obtained by averaging over 32 MC simulations.

\subsection{Compared methods and their implementation details}
We compared the proposed algorithms with the state of the art randomized preconditioning techniques which can reach any level of desired accuracy within a bounded number of iterations. In the conducted comparison study we used a total of 5 previously proposed techniques that can be briefly described as follows. The Blendenpik uses the $R$ matrix in the QR decomposition of the sketched matrix $SA$ as the preconditioning matrix for the LSQR algorithm just like the method proposed by Rokhlin et al. \cite{ref:blendenpik, ref:rokhlin} and it uses Randomized Orthonormal System (ROS) to generate the sketched matrix \cite{ref:pilanci1}. The LSRN uses the $V$ matrix in the SVD similar to the Blendenpik. In spite of its high running complexity, for parallelization purposes, the Gaussian sketch matrices are preferred in the LSRN. In addition to the LSQR, also the CS can be preferred in the LSRN as the core solver in distributed computational environments \cite{ref:lsrn}. The IHS uses the sketched Hessian as the preconditioning matrix for the Gradient Descent. The Accelerated IHS (A-IHS) uses this idea for the CG algorithm in over-determined problems. The dual counter-part of the A-IHS algorithm, A-IDRP, is shown to be faster than the Dual Random Projection algorithm proposed in \cite{ref:dual_rp}, so we did not include the DRP in the simulations. Additionally, we include a CS variant of the IHS (IHS-CS) to the comparisons: we combined the randomized preconditioning idea of the IHS with the preconditioned CS method \cite{ref:templates}. We found the bounds for the eigenvalues in the same way as in the LSRN. We have solved the low dimensional sub-problems required by all the IHS variants by taking the QR decomposition, but for \textit{inexact} schemes, we have used the proposed \ff{AAb\_Solver} with a constant forcing term. Although the inexact approach is also applicable for the accelerated algorithms proposed in \cite{ref:acc_ihs}, we did not include them in the simulations since their exact versions are outperformed by the \textit{Exact} \ff{M-IHS} variants in all settings. Except for the LSRN variants which use Gaussian sketch matrices, we used Discrete Cosine Transform in the ROS for all the compared techniques.

\subsection{Linear systems with noiseless measurements}
In the first experiment, we did not include noise in the linear system to emphasize the convergence rate that the algorithms can provide in such severely ill posed problems. To make the problem more challenging, for this experiment only we sampled the input vector $x_0$ from uniform distribution Uni$(-1, 1)$. In such scenarios, convergence rates of Krylov subspace-based iterative solvers without preconditioning fall to its minimum value since the energy of the input is distributed equally over the range space of $A$. The obtained results are shown in Figure \ref{fig:no_reg}. % 
\input{demos/fig_no_reg.tex}%
Due to high running time of the Gaussian sketches, $O(mnd)$, the LSRN variants require more operations (for the size of the problems considered here approximately 10 times larger) than the others. Due to the lack of inner product calculations, the \ff{M-IHS} requires slightly fewer operations than the Blendenpik, nonetheless, it reaches to the same accuracy with the LSRN-LSQR. The A-IHS algorithm has the worst performance which is expected in the un-regularized problems, since it is adapted on the CG technique that can be unstable for the un-regularized LS problems due to the high condition number \cite{ref:lsqr}. The convergence of the CS-based techniques, both of the IHS and the LSRN variants, are substantially slower than the \ff{M-IHS}, which suggests that the \ff{M-IHS} algorithm can take the CS’s place in those applications where parallel computation is an option. A similar comparison of the \ff{M-IHS} with the Accelerated Randomized Kaczmarz (ARK) and the CGLS without preconditioning has been shown in Figure 2 of \cite{ref:mihs}.
\subsection{Linear systems with noisy measurements}
We tested robustness of the methods against noise on regularized LS problems by using an additive i.i.d. Gaussian noise at level $\lscost{\omega}{2}\big/\lscost{Ax_0}{2} = 1\%$. For this purpose, the optimal regularization parameter that minimizes the error $\lscost{x^*(\lambda) - x_0}{2}$ is provided to all techniques. Each technique is allowed to conduct a total of 20 iterations. Results for strongly over-determined and strongly under-determined cases can be seen in Figures \ref{fig:over_reg} and \ref{fig:under_reg}, respectively. %
\input{demos/fig_over_reg.tex}%
\input{demos/fig_under_reg.tex}%
We used a sketch size of $m = \min(n,d)$ to emphasize the promise of the RP techniques although such sizes are not applicable for the LSRN variants. Even if the sketch size has been increased further, the convergence of the LSRN variants were considerably slower than the others; so we leave out the LSRN variants from the comparison set in the regularized settings. Also, in the regularized setup, the A-IHS and A-IDRP methods are slower than the Blendenpik, IHS-CS and \ff{M-IHS} variants. Besides, the inexact schemes proposed for the \ff{M-IHS} and \ff{Dual M-IHS} require significantly less operations to reach to the same level of accuracy as their exact versions. Although the inexact schemes require approximately 10 times less operations then their exact versions in these setups; the saving gets larger as the sketch size increases as examined in Section \ref{sec:varying}, because while any full decomposition requires $O(mr^2)$ operations, approximately solving the sub-problem requires only $O(mr)$ operations. 

As long as the statistical dimension of the problem is small with respect to the dimensions of coefficient matrix $A$, Lemma \ref{corollary:sketch} implies eligibility of sketch sizes that are smaller than the rank, $m\leq \min(n,d)$. This implication can be verified in Figure \ref{fig:square_reg} on which we showed the performance of the \ff{Primal Dual M-IHS} techniques. %
\input{demos/fig_square_reg.tex}%
Here, the \textit{inexact} schemes of the \ff{M-IHS} and \ff{Dual M-IHS} use a sketch size $m = 2\cdot\sd(A)$. The primal dual schemes use $m_1 = m_2 = 2\cdot\sd(A)$ except for the \ff{Primal Dual M-IHS} shown as a green curve which uses $m_1 = m_2 = 8\cdot\sd(A)$. All the methods are allowed to conduct $N = 60$ iterations except the \ff{Primal Dual M-IHS} with larger sketch size is allowed to conduct only $20$ iterations. The number of inner iterations are restricted by $M = 25$ for all the primal dual schemes. Lastly, a fixed forcing term $\epsilon_{sub} = 0.1$ is used in the \ff{AAb\_Solver} for all the inexact schemes. Applying a second dimension reduction may not seem to create significant computational saving, but this approach produces smaller sub-problems than the \ff{M-IHS} and the \ff{Dual M-IHS} techniques therefore enables estimation of parameters such as $\lambda$ with far fewer number of operations. Lastly, the \ff{Primal Dual M-IHS} variants have a noticeably higher rate of convergence than the A-IPDS algorithm which is based on the CG technique. 

\subsection{Scalability to larger size problems}\label{sec:varying}
In this section, as the size of the coefficient matrix and the sketch size increase we show that the saving gained by the inexact schemes become critically more important. For this purpose, the algorithms were run on the over-determined problems with size $5\cdot10^4\times\gamma\cdot500$ where $\gamma\in\{1,\ 2,\ 4,\ 8,\ 16\}$. The sketch size was chosen as $m=d=\gamma\cdot500$ and the regularization parameter was set to $0.1453$ for all the experiments so that $\sd(A) = d/10$ remains the same for all the experiments. The data was generated by using the setup described in Section \ref{sec:ex_setup}. Note that the convergence properties of the proposed techniques depend only on the statistical dimension but not directly to the decay rate of the singular values. To show this, for these experiments, we used \textit{heat} singular value profile that has significantly lower decay rate than the \textit{philips} profile used earlier. The experiments were realized on a desktop with 4Ghz i7-4790K CPU processor and 32Gb RAM. The flop count and wall clock time of the algorithms to reach to an $(\eta=10^{-4})$-optimal solution approximation are shown in Figure \ref{fig:tot_varying_d}. %
\input{demos/fig_varying_d_tot.tex}%
\input{demos/fig_varying_d_detail.tex}%
As $d$ and $m$ reach thousands, the number of operations required by the exact schemes (Blendenpik and \ff{M-IHS}) becomes larger than $100$ times of the operation count required by the inexact scheme. Moreover, the exact schemes need $25$ time longer time than the inexact scheme to reach the desired accuracy. Additionally, the operation counts and elapsed time in each stage of the algorithms can be seen in Figure \ref{fig:detail_varying_d} which shows that even the cost of the decomposition applied on the sketched matrices reaches to prohibitive levels for large scale problems. Hence the use of solvers such as \ff{M-IHS} variants that allow inexact schemes is the only practical choice in these regimes. In these experiments, for the estimation of the statistical dimension, we set $T = 2$ and $\epsilon_{tr} = 0.5$. The additional cost of the $\sd(A)$ estimation for the proposed \ff{M-IHS} variants becomes negligibly small when $R$-factor is utilized; for the inexact schemes, still it has a low cost, around the cost of one M-IHS-inexact iteration, that does not cause an issue unlike a matrix decomposition.
\begin{remark}
Bench-marking of the exact and inexact schemes by using wall clock time in MATLAB is not a fair comparison because for-loops in the interpreted languages such as MATLAB is well known to be much slower than the loops in compiled languages such as C. Most of the decompositions in MATLAB have C-based implementation with professional use of BLAS operations, while the inexact schemes are based on a for-loop. Therefore, we prefer to rely on the operation counts. However, to give an opinion, in spite of the disadvantages we demonstrate timing as well.
\end{remark}
\subsection{Effect of the statistical dimension on the performance of the inexact schemes}
The inexact schemes become more efficient as the statistical dimension decreases since the sub-problems are solved in less iterations. To show the effect of varying statistical dimension on the complexity of the algorithms, we used over-determined problems with size $5\cdot10^4\times16\cdot10^3$ and varied the regularization parameter to obtain different $\rho = \sd(A)/d$ ratios where $\rho \in \{0.5\%,\ 1\%, \ 2\%,\ 5\%,\ 10\%, \  20\%,\ 50\%\}$. The sketch size was chosen as $m=d$ and \textit{heat} profile was used. As in Section \ref{sec:varying}, the flop count and wall clock time of the algorithms to reach to an $(\eta=10^{-4})$-optimal solution approximation for the problems with different statistical dimensions are shown in Figure \ref{fig:tot_varying_sd}. %
\input{demos/fig_varying_sd_tot.tex}%
\input{demos/fig_varying_sd_detail.tex}%
Complexity of the exact schemes increases by the increasing statistical dimension since convergence rate $\sd(A)/m$ decreases as $m$ remains constant. Complexity of the inexact scheme increase faster since sub-problems require more iterations as the effective range space gets larger. The effect of the increasing statistical dimension over the different stages of the algorithms are shown in Figure \ref{fig:detail_varying_sd}.  For the estimation of $\sd(A)$, same parameters $T = 2$ and $\epsilon_{tr} = 0.5$ were used as early. Even for the large $\rho$ ratios, utilizing a sub-solver is still more effective than computing a matrix decomposition.

%% file: demos/fig_no_reg.tex
\begin{figure}[!hptb]
\centering
  \includegraphics[width=1\linewidth]{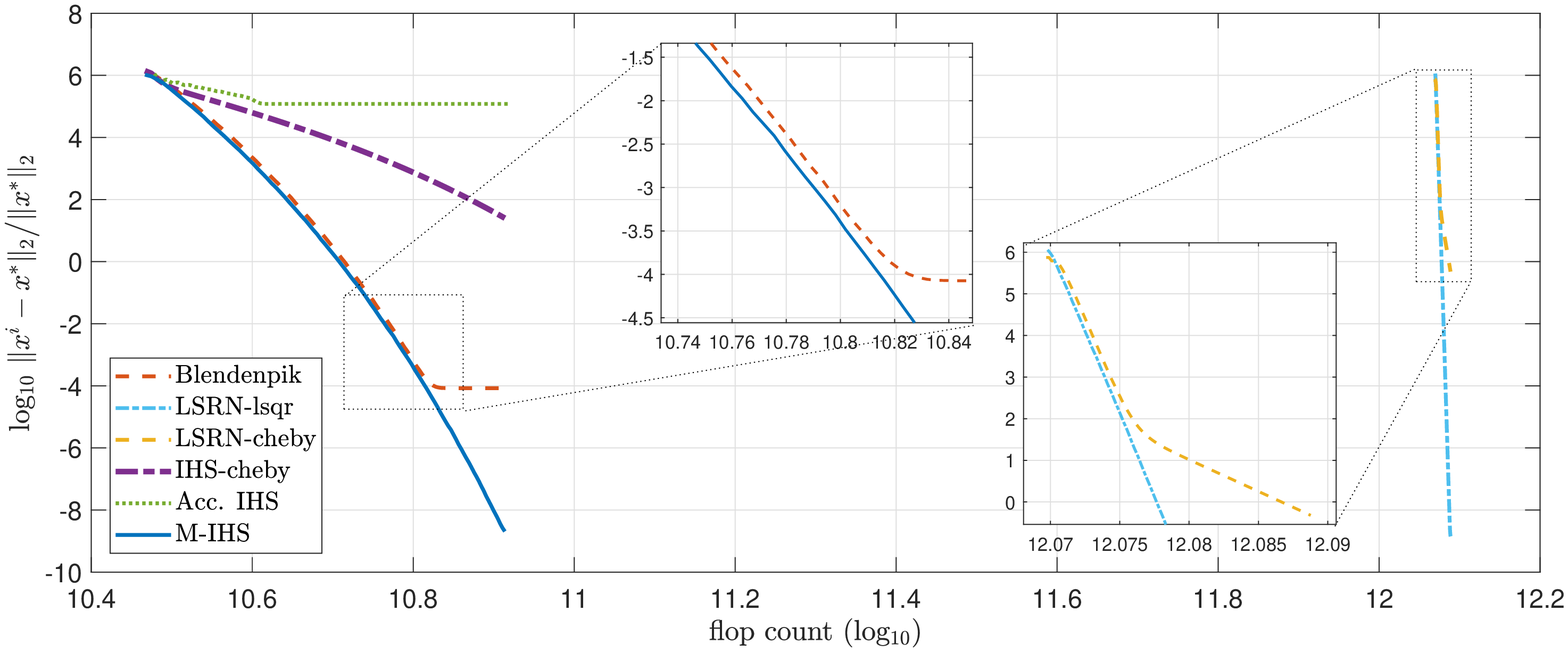}
  \caption{{{Performance comparison on an un-regularized LS problem with size $2^{16}\times 2000$. In order to compare the convergence rates, number of iterations for all solvers are set to $N = 100$ with the same sketch size: $m = 4000$. According to Corollary \ref{corollary:empirical}, we expect the \ff{M-IHS} to reach an accuracy: $\lscost{x^N - x_0}{2} \leq \kappa(A)\lscost{x_0}{2} \left({1}\big/{\sqrt{2}}\right)^{N} = 9\cdot 10^{-8}$, which closely fits to the observed case.}}}
  \label{fig:no_reg}
\end{figure}%

%% file: demos/fig_over_reg.tex
\begin{figure}[ht]
\centering
  \includegraphics[width=1\linewidth]{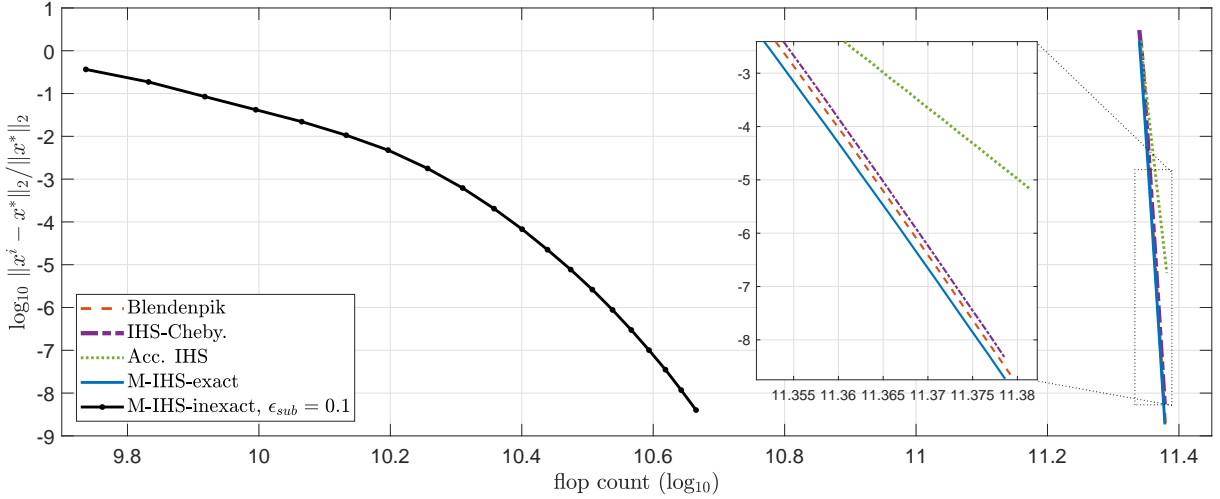}
  \caption{{Performance comparison on a regularized LS problem ($n\gg d$) with dimensions $(n,d,m, \sd(A)) = (2^{16}, 4000, 4000, 443)$. According to Corollary \ref{corollary:empirical}, \ff{M-IHS} is expected to satisfy: $\lscost{x^N - x^*}{2} \leq \lscost{x^*}{2}\sqrt{\kappa(A^TA + \lambda I_d)}\left(\sqrt{443/4000}\right)^N = 6\cdot10^{-9}$ which is almost exactly the case. The \textit{Inexact} \ff{M-IHS} requires significantly fewer operations to reach the same accuracy as others. For example to obtain an $(\eta = 10^{-4})$-optimal solution approximation, the \textit{Inexact} \ff{M-IHS} requires approximately 10 times less operations than any techniques that need factorization or inversion of the sketched matrix.}}
  \label{fig:over_reg}
 
\end{figure}

%% file: demos/fig_under_reg.tex
\begin{figure}[h]
\centering
  \includegraphics[width=1\linewidth]{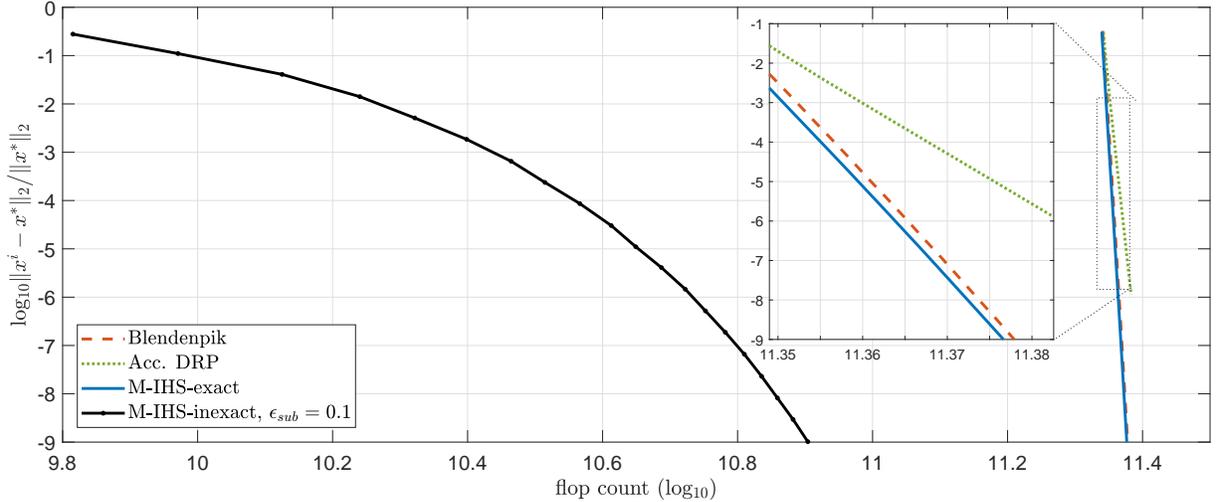}
  \caption{{Performance comparison on a regularized LS problem ($n \ll d$) with dimensions $(n,d,m,\sd(A)) = (4000, 2^{16}, 4000, 462)$. The comments in Figure \ref{fig:over_reg} are also valid for this case. The \textit{Inexact} scheme for \ff{Dual M-IHS} is capable of significantly reducing the complexity.}}
  \label{fig:under_reg}
  
\end{figure}

%% file: demos/fig_square_reg.tex
\begin{figure}[h]
\subfigure[{$n\geq d$ and $\sd(A) = 680$}]{\includegraphics[width=0.5\linewidth]{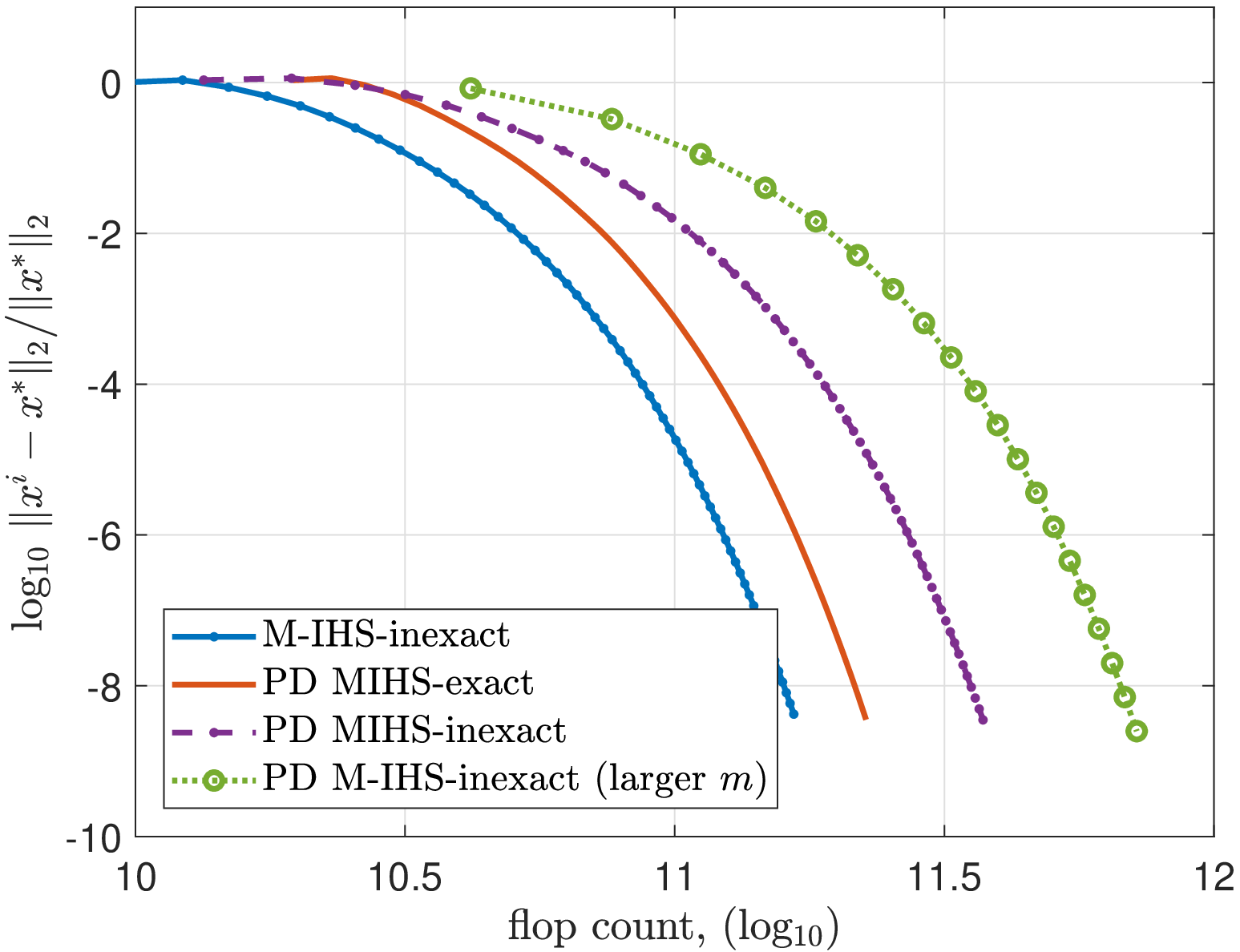}}%
\subfigure[{$n\leq d$ and $\sd(A) = 825$}]{\includegraphics[width=0.5\linewidth]{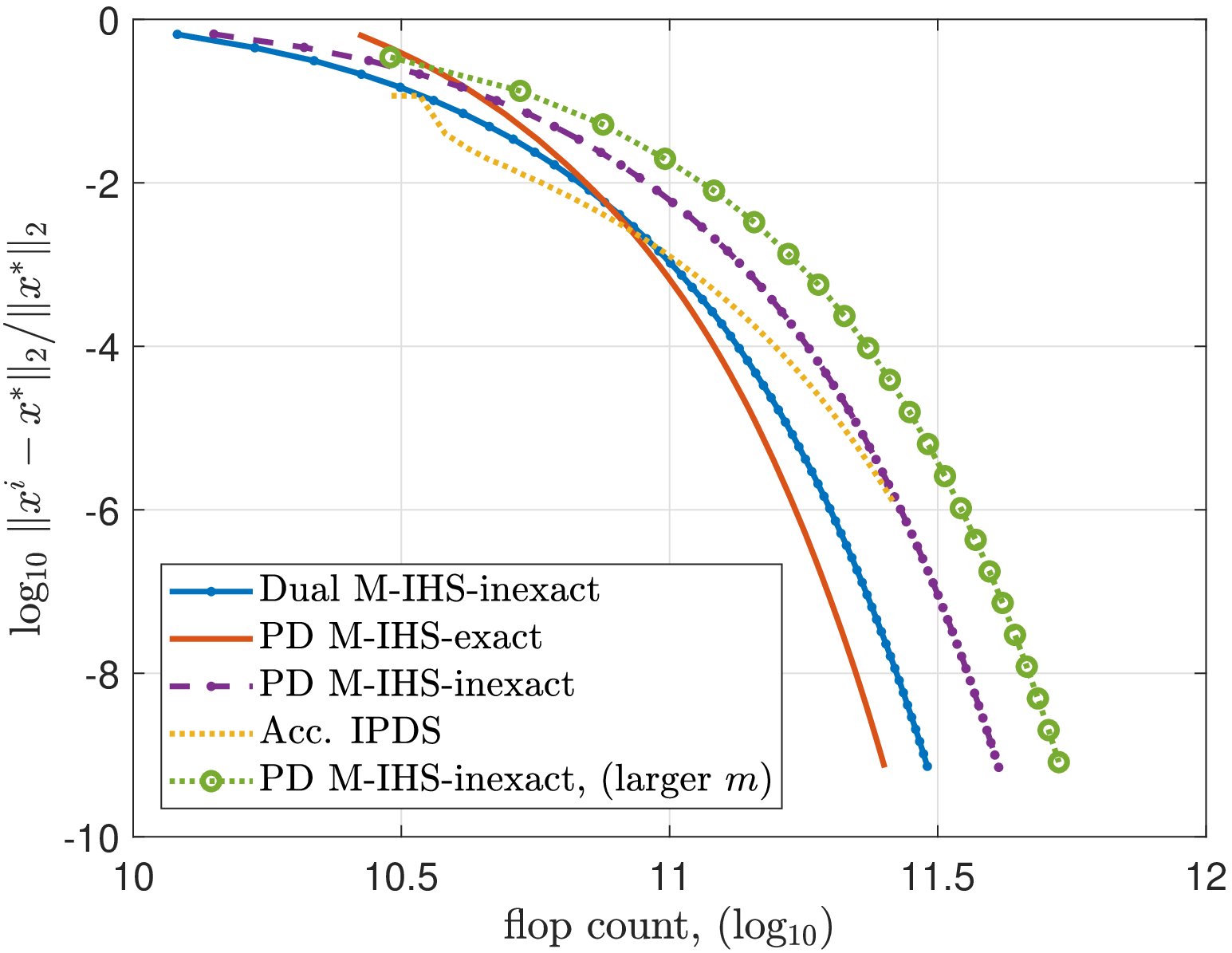}}
\caption{{Performance comparison on regularized LS problems with square-like dimensions. The problem dimensions are set to $\max(n,d) = 5\cdot10^4$ and $\min(n,d) = 10^4$ with a noise level of $10\%$. The results verifies two hypothesis: first the sketch size for the \ff{M-IHS} variants can be chosen proportional to the statistical dimension even if it becomes smaller than the size of the coefficient matrix. Second, the coefficient matrix can be sketched from both sides to reduce computational complexity.}}
\label{fig:square_reg}

\end{figure}

%% file: demos/fig_varying_d_tot.tex
\begin{figure}[h]
\subfigure{\includegraphics[width=0.5\linewidth]{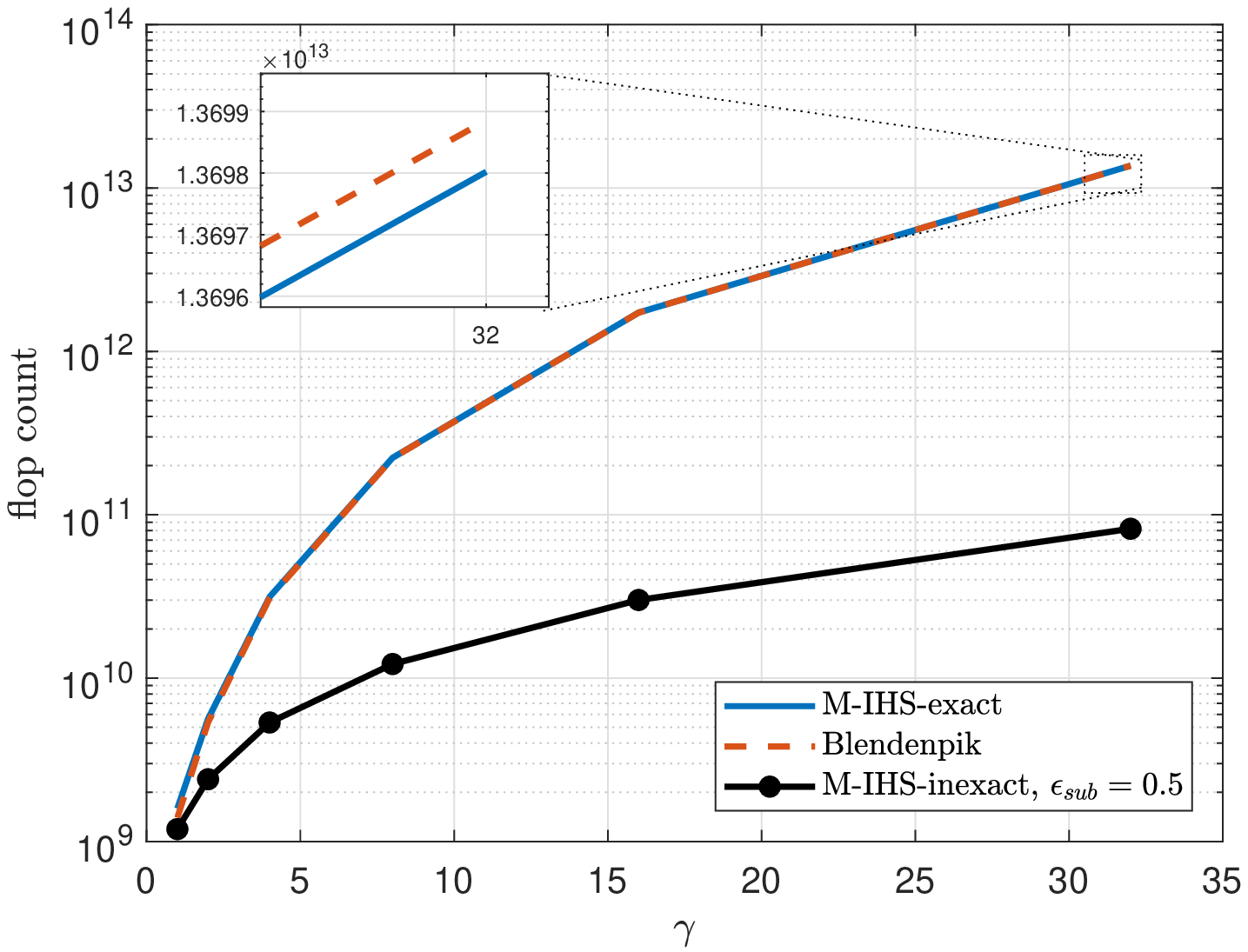}}%
\subfigure{\includegraphics[width=0.5\linewidth]{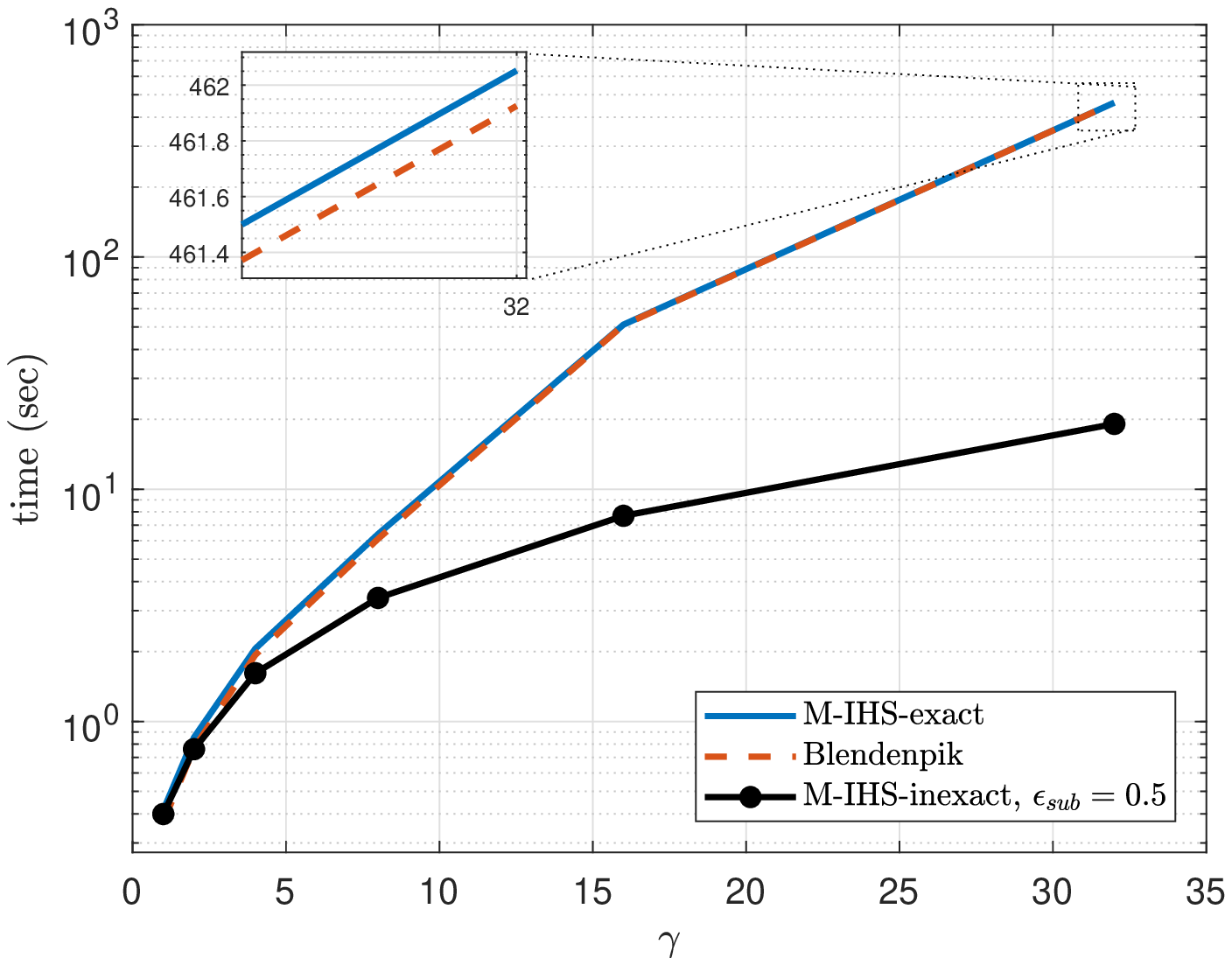}}%
\caption{{Complexity of the algorithms in terms of operation count and computation time on a set of $5\cdot10^4\times500\cdot\gamma$ dimensional over-determined problems with $m = d$ and $\sd(A) = d/10$.}}
  \label{fig:tot_varying_d}
   
\end{figure}

%% file: demos/fig_varying_d_detail.tex
\begin{figure}[h]
\subfigure{\includegraphics[width=0.5\linewidth]{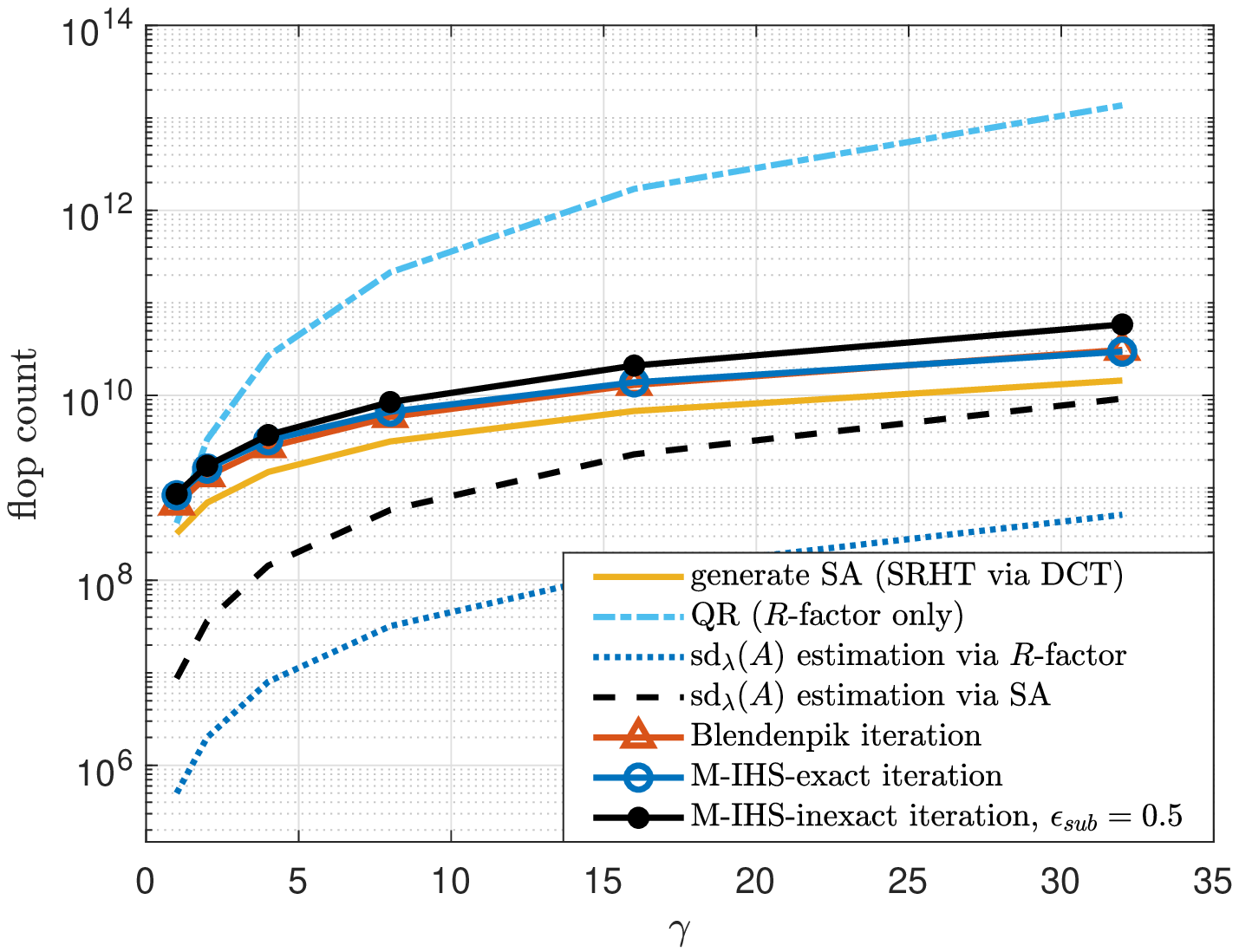}}%
\subfigure{\includegraphics[width=0.5\linewidth]{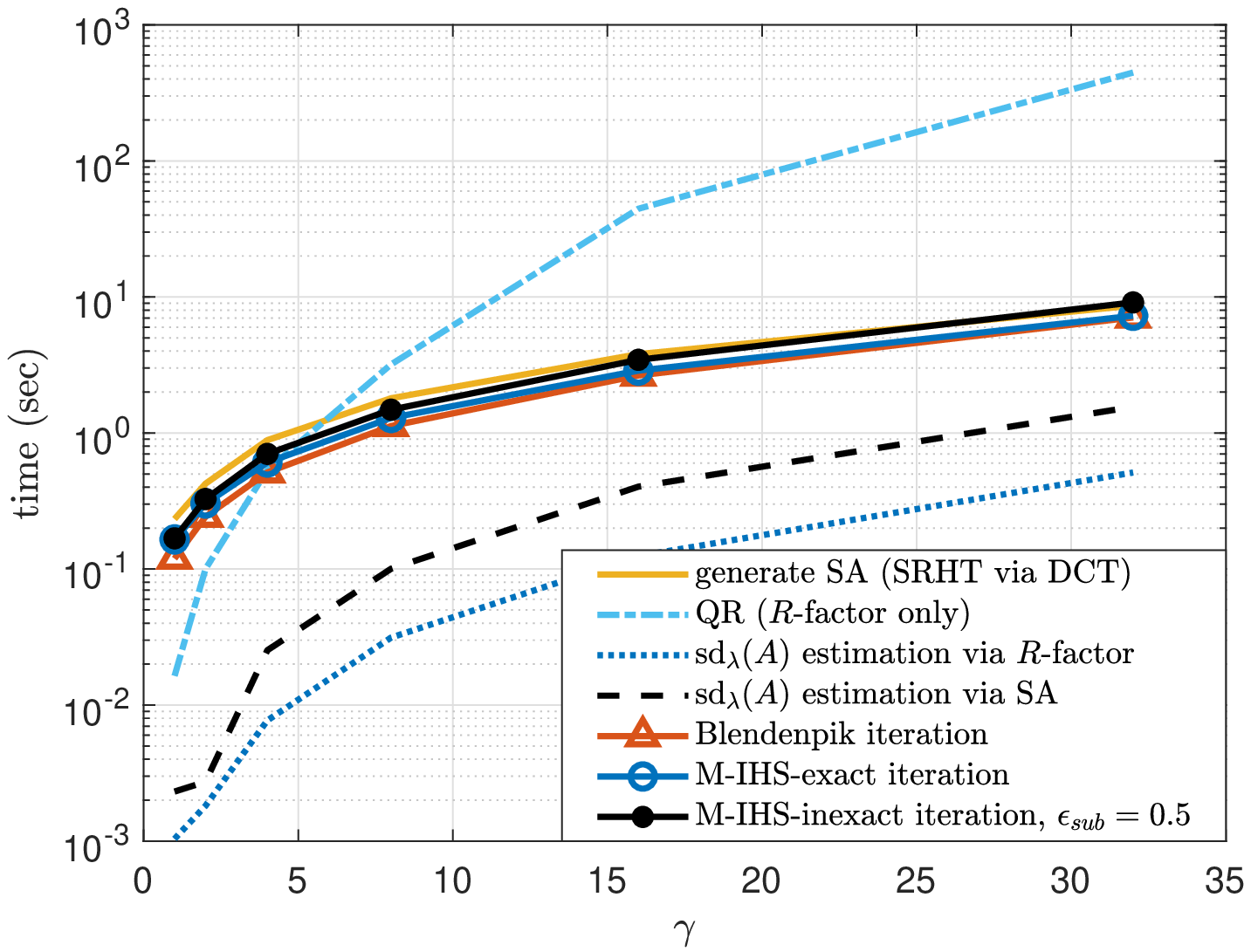}}
\caption{{Complexity of the each stage in terms of operation count and computation time on a set of $5\cdot10^4\times500\cdot\gamma$ dimensional over-determined problems with $m=d$ and $\sd(A) = d/10$. All methods contain SA generation stage. The Blendenpik and M-IHS-exact contain also QR decomposition stage but M-IHS-inexact does not. The M-IHS-exact estimates $\sd(A)$ by using the $R$-factor while the M-IHS-inexact uses directly $SA$ matrix and the \ff{AAb\_Solver} as proposed in Algorithm \ref{algo:trace}. The results show that the matrix decompositions are the main computational bottleneck for the exact schemes in large scale problems where the advantage of the inexact schemes becomes more significant.}}
\label{fig:detail_varying_d}

\end{figure}

%% file: demos/fig_varying_sd_tot.tex
\begin{figure}[h]
\subfigure{\includegraphics[width=0.5\linewidth]{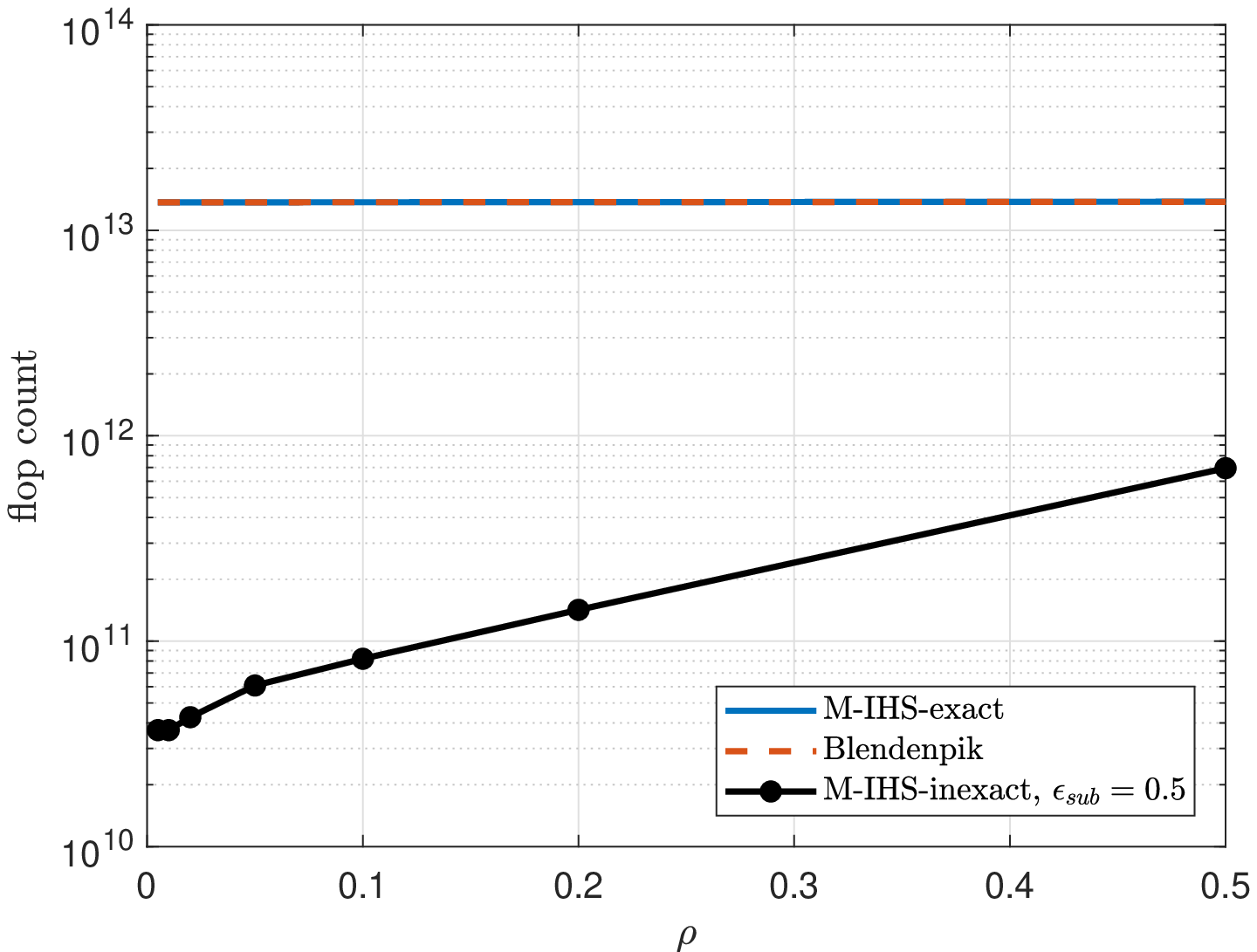}}%
\subfigure{\includegraphics[width=0.5\linewidth]{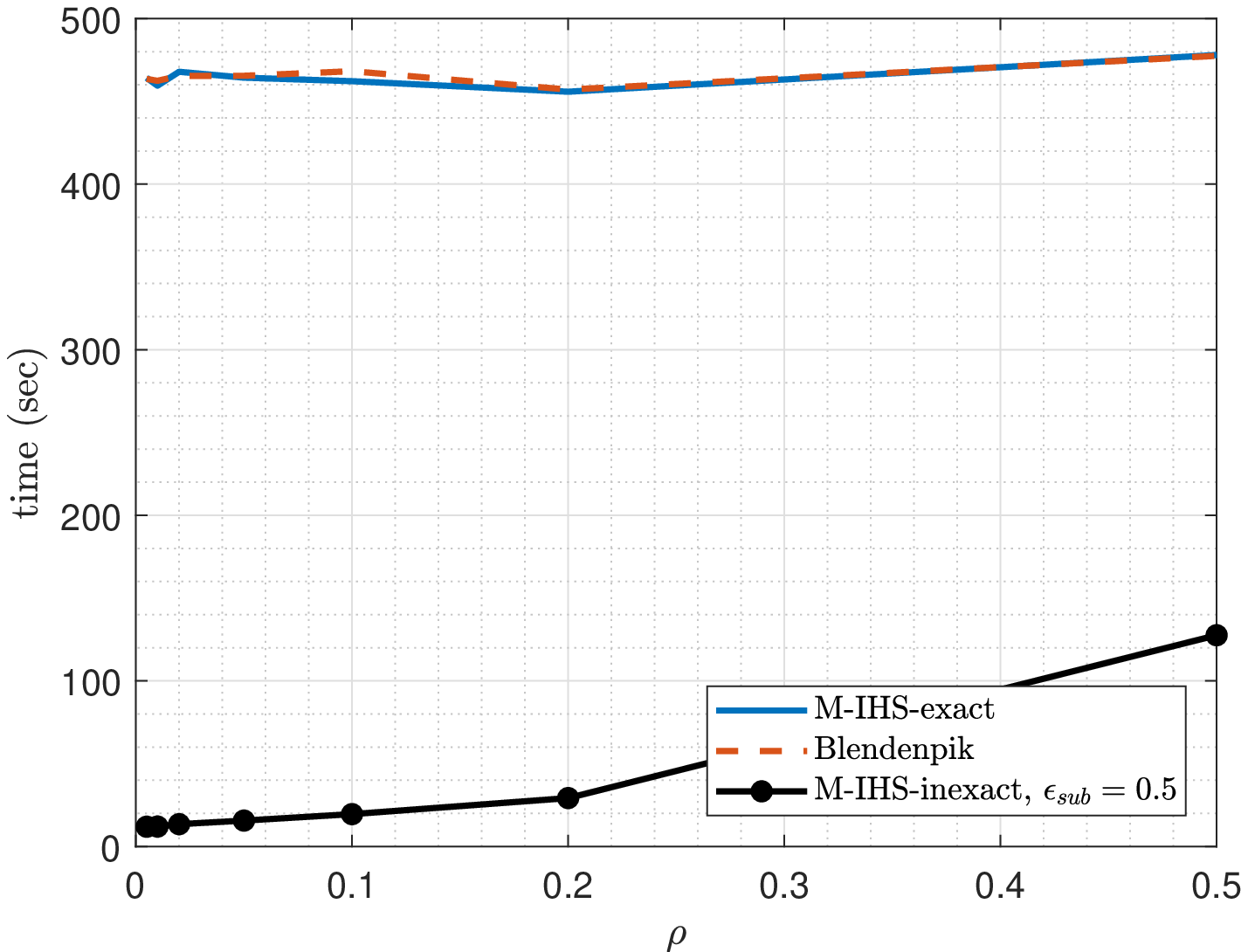}}%
\caption{{Complexity of the algorithms in terms of operation count and computation time on a $5\cdot10^4\times4\cdot10^3$ dimensional problem for different $\rho=\sd(A)/d$ ratios.}}
\label{fig:tot_varying_sd}
\end{figure}

%% file: demos/fig_varying_sd_detail.tex
\begin{figure}[!htbp]
\subfigure{\includegraphics[width=0.5\linewidth]{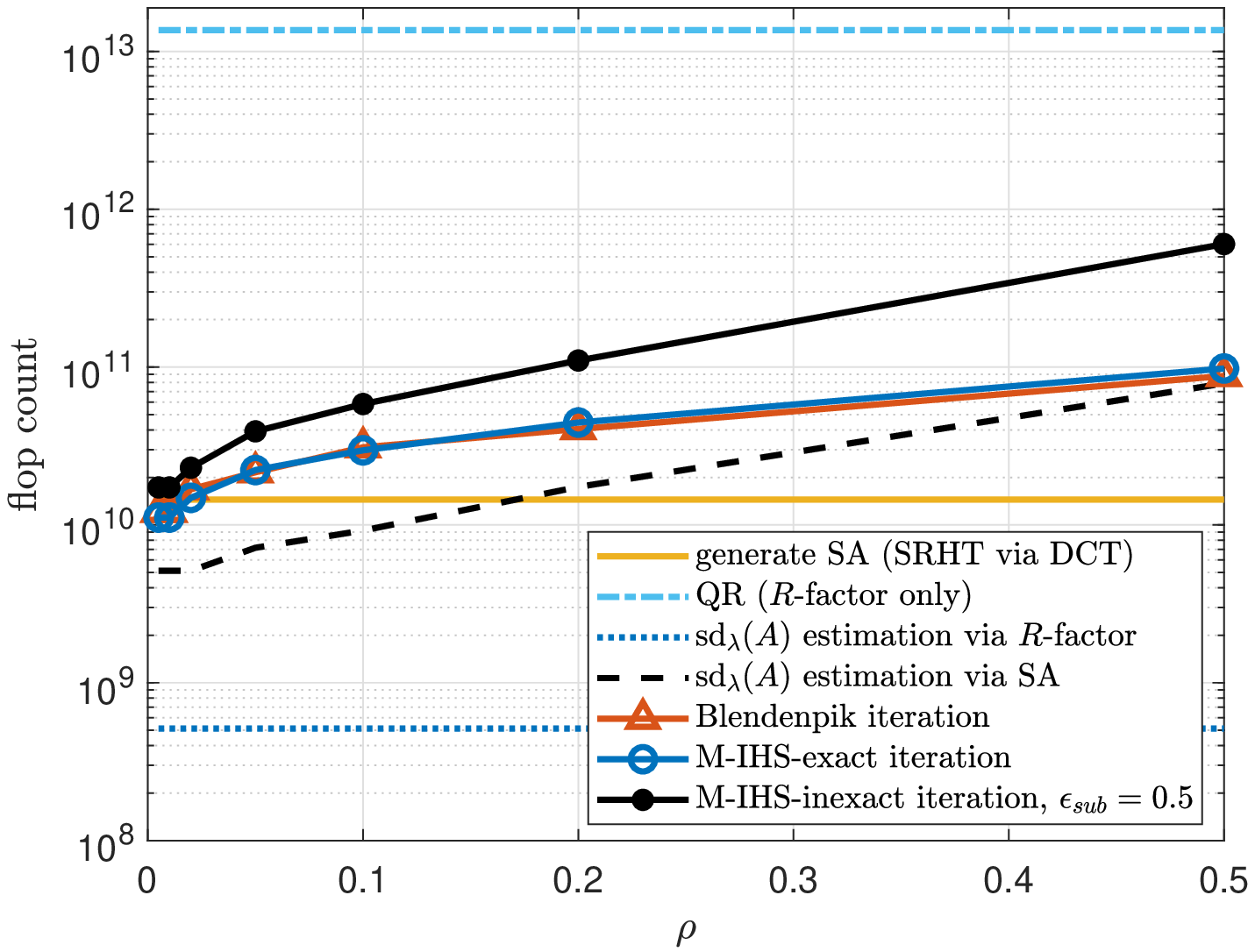}}%
\subfigure{\includegraphics[width=0.5\linewidth]{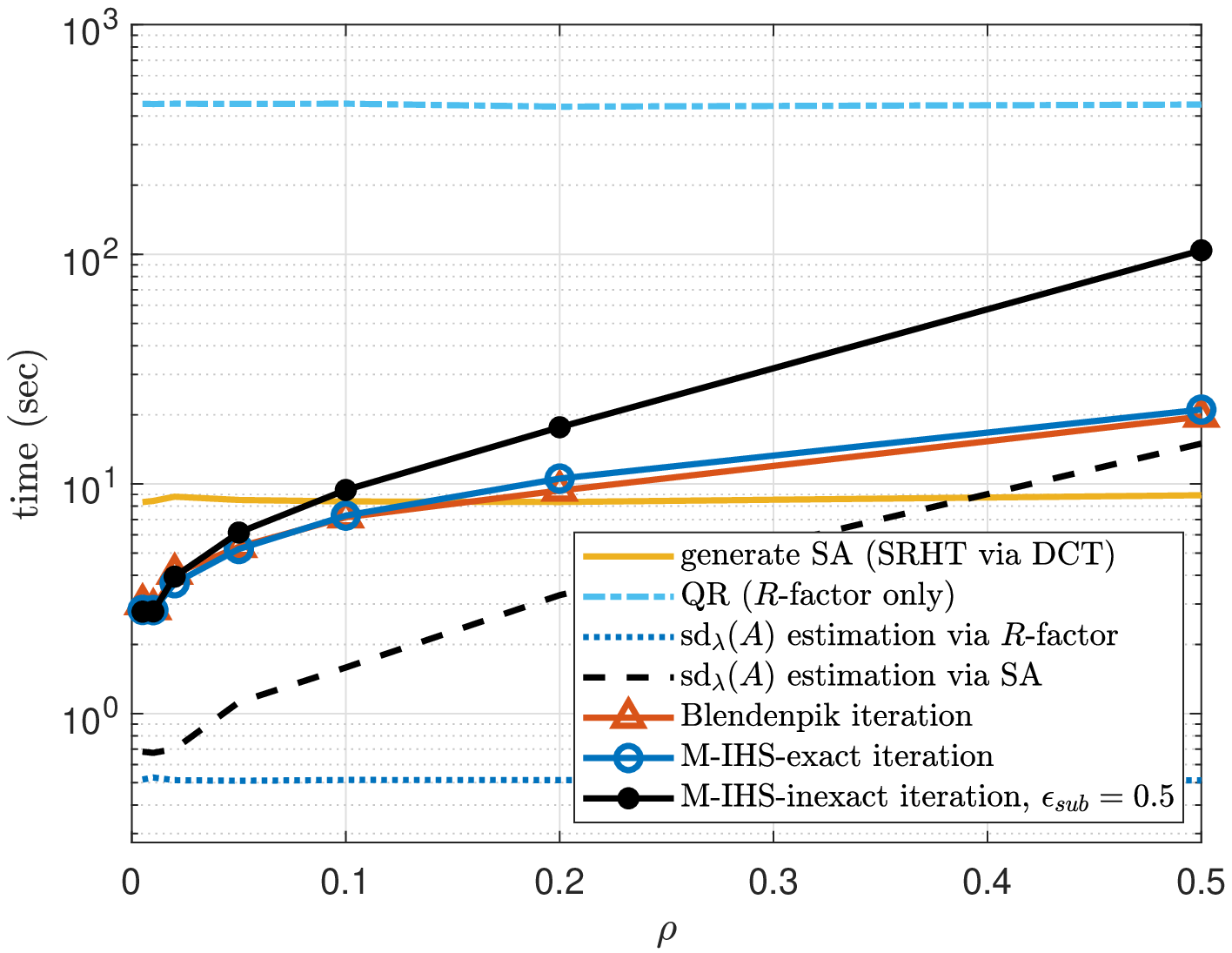}}%
\caption{{Complexity of each stage in terms of operation count and computation time on a $5\cdot10^4\times4\cdot10^3$ dimensional problem for different $\rho=\sd(A)/d$ ratios. Stages of each algorithm is given in Figure \ref{fig:detail_varying_d}.}}
\label{fig:detail_varying_sd}

\end{figure}

%% file: sections/conclusion.tex
Based on the IHS and the Heavy Ball Acceleration we proposed and analyzed the \ff{M-IHS} solvers for large scale LS problems. We examined the effect of $\ell2$ norm regularization on the optimal momentum parameters. We obtained lower bounds on the sketch size for several randomized distributions in order to get a pre-determined convergence rate with a constant probability. The bounds suggest that the sketch size can be chosen proportional to the statistical dimension of the problem. Hence, the \ff{M-IHS} variants can be used for any dimension regimes if the statistical dimension is sufficiently smaller than the dimensions of the coefficient matrix. We empirically showed the ratio between statistical dimension and the sketch size determines the convergence rate of the proposed \ff{M-IHS} variants. The main advantage of the proposed \ff{M-IHS} variants over the state of the art randomized preconditioning techniques such as the Blendenpik, A-IHS and LSRN is their ability to use inexact schemes that avoids matrix decompositions or inversions. As demonstrated in a wide range of numerical experiments, computational saving provided by the proposed solver becomes decidedly significant in large scale problems. Lastly, the proposed \ff{M-IHS} variants avoid using any inner products in their iterations and they are shown to be faster than CS-based randomized preconditioning algorithms. Therefore the proposed \ff{M-IHS} variants are strong candidates to be the techniques of choice in parallel or distributed architectures.

%% file: sections/inaccuracies.tex
In this appendix, we provide details of a critical discussion on the steps of the derivation that leads to an error upper bound for the iterations of the primal dual algorithms given in \cite{ref:acc_ihs}. First, we provide a short list of minor issues that can easily be corrected.  
\begin{enumerate}
    \item During the initialization stage in \textit{Line 2} of both \textit{Algorithm 4} in page 4097 and \textit{Algorithm 5} in page 4098, the residual error vector $\mathbf{r}^{(0)}$ must be set to $-\lambda\mathbf{y}$ instead of $\mathbf{-y}$, otherwise iterates of the both of the algorithms diverge from the optimal solution. 
    \item During the initialization stage in \textit{Line 2} of \textit{Algorithm 7} in 4912, the dual residual error vector $\mathbf{r}_\text{Dual}^{(0)}$ must be set to $\mathbf{-\lambda\mathbf{y}}$ instead of $\mathbf{-y}$ and during the initialization stage of the inner loop iterations in \textit{Line 15}, the primal residual error vector $\mathbf{r}_\text{P}^{(0)}$ must be set to $-\mathbf{R}^T\mathbf{X}^T\mathbf{r}_\text{D}^{(t+1)}$; otherwise iterates of the algorithm diverges from the optimal solution. The MATLAB codes provided in the \href{https://github.com/ibrahimkurban/M-IHS}{link} includes these corrections.
\end{enumerate}
In addition to the above mentioned minor issues, there are some major issues as well. Unfortunately, we could not obtain corrective actions on these major issues as we could have done on the minor issues mentioned above. Therefore, a lower bound on the number of inner loop iterations, that guarantee a certain rate of convergence at the main loop, is still an open question for the primal dual algorithms. In the remaining of this appendix, we will provide steps of the derivation presented in \cite{ref:acc_ihs}, along with our critical remarks on their validity.

Consider the A-IHS updates $ \widehat{\mathbf{w}}^{t+1} = \widehat{\mathbf{w}}^t + \widehat{\mathbf{u}}^t.$ We are going to use exactly the same notation as \cite{ref:acc_ihs} except for that \textit{HS} subscript for the A-IHS iterates are omitted. In the primal dual algorithms, instead of exact sequence $\{\widehat{\mathbf{w}}^{t}\}$, a sequence $\{\widetilde{\mathbf{w}}^{t}\}$ is obtained due to the approximate minimizers that are used in place of $\widehat{\mathbf{u}}^t$. Consequently while the sequence $\{\widehat{\mathbf{w}}^{t}\}$ is obtained after $t$ exact iterations of the A-IHS algorithm, sequence $\{\widetilde{\mathbf{w}}^{t}\}$ is obtained after $t$ primal dual iterations in each of which $k$ inner loop updates are used to approximate $\widehat{\mathbf{u}}^t$'s. The details of the inner and outer loops can be found in Algorithm 7 of \cite{ref:acc_ihs}. The aim of the \textit{Theorem 9} is to establish an upper bound for $\lscost{\widetilde{\mathbf{w}}^{t+1} - \mathbf{w}^*}{\mathbf{X}}$ where $\mathbf{w}^*$ is the true minimizer of the primal objective function. The triangle inequality and the convergence rate of the A-IHS that is established in \textit{Theorem 2} of \cite{ref:acc_ihs} is used to find an upper bound for this error norm:
\begin{align}
    \lscost{\widetilde{\mathbf{w}}^{t+1} - \mathbf{w}^*}{\mathbf{X}} &\leq \lscost{\widehat{\mathbf{w}}^{t+1} - \mathbf{w}^*}{\mathbf{X}} + \lscost{\widetilde{\mathbf{w}}^{t+1} - \widehat{\mathbf{w}}^{t+1}}{\mathbf{X}}\leq \alpha^t\lscost{\mathbf{w}}{\mathbf{X}} + \lscost{\widetilde{\mathbf{w}}^{t+1} - \widehat{\mathbf{w}}^{t+1}}{\mathbf{X}},\nonumber
\end{align}
where $\alpha = \frac{C_0\sqrt{\mathbb{W}^2(\mathbf{X}\reals{p}\cap\mathcal{S}^{n-1}})\log(1/\delta)}{1-C_0\sqrt{\mathbb{W}^2(\mathbf{X}\reals{p}\cap\mathcal{S}^{n-1}})\log(1/\delta)}$. At this point a new iterate, $\widebar{\mathbf{w}}^{t+1}$, is introduced, which is the result of one exact step of the IHS initialized at $\widetilde{\mathbf{w}}^{t}$. The inner loop iterations at the $t$-th outer (main) loop iteration of the primal dual iterations are expected to converge $\widebar{\mathbf{w}}^{t+1}$. Therefore,
\begin{align}
    \lscost{\widetilde{\mathbf{w}}^{t+1} - \widehat{\mathbf{w}}^{t+1}}{\mathbf{X}} &\leq \lscost{\widetilde{\mathbf{w}}^{t+1} - \widebar{\mathbf{w}}^{t+1}}{\mathbf{X}} + \lscost{\widebar{\mathbf{w}}^{t+1} - \widehat{\mathbf{w}}^{t+1}}{\mathbf{X}},\nonumber\\
    \lscost{\widetilde{\mathbf{w}}^{t+1} - \widebar{\mathbf{w}}^{t+1}}{\mathbf{X}} &\leq \lambda_{max}\left(\frac{\mathbf{X}^T\mathbf{X}}{n}\right)\beta^k \lscost{\widebar{\mathbf{w}}^{t+1}}{2} \leq \lambda_{max}\left(\frac{\mathbf{X}^T\mathbf{X}}{n}\right)\beta^k \left(\lscost{\widebar{\mathbf{w}}^{t+1} - \mathbf{w}^*}{2} + \lscost{\mathbf{w}^*}{2}\right) \nonumber\\
    &\leq 2\lambda_{max}\left(\frac{\mathbf{X}^T\mathbf{X}}{n}\right)\beta^k  \lscost{\mathbf{w}^*}{2}, \label{wrong:inner}
\end{align}
where $\beta = \frac{C_0\sqrt{\mathbb{W}^2(\mathbf{X}^T\reals{p}\cap\mathcal{S}^{p-1}})\log(1/\delta)}{1-C_0\sqrt{\mathbb{W}^2(\mathbf{X}\reals{p}\cap\mathcal{S}^{p-1}})\log(1/\delta)}$. The last inequality is not valid unless $\lscost{\widebar{\mathbf{w}}^{t+1} - \mathbf{w}^*}{2} \leq \lscost{\mathbf{w}^*}{2}.$ However, particularly during the initial phases of the main iterations this condition can be violated. Therefore, this step of the proof requires a major revision. Assuming that such revision is possible, up to this point, the following is obtained:
\begin{equation}
    \lscost{\widetilde{\mathbf{w}}^{t+1} - \mathbf{w}^*}{\mathbf{X}} \leq \alpha^t\lscost{\mathbf{w}}{\mathbf{X}} + 2\lambda_{max}\left(\frac{\mathbf{X}^T\mathbf{X}}{n}\right)\beta^k  \lscost{\mathbf{w}^*}{2} + \lscost{\widebar{\mathbf{w}}^{t+1} - \widehat{\mathbf{w}}^{t+1}}{\mathbf{X}}.\label{eq:C}
\end{equation}
To proceed for the final form of the upper bound, the following upper bound on the last term of \eqref{eq:C} is given in \cite{ref:acc_ihs}:
\begin{equation*}
    \lscost{\widebar{\mathbf{w}}^{t+1} - \widehat{\mathbf{w}}^{t+1}}{\mathbf{X}} \leq \lscost{\widetilde{\mathbf{H}}^{-1}}{2}\lscost{\widetilde{\mathbf{H}} - \mathbf{H}}{2}\lscost{\widetilde{\mathbf{w}}^{t} - \widehat{\mathbf{w}}^{t}}{\mathbf{X}} \leq  \frac{4\lambda_{max}\left(\frac{\mathbf{X}^T\mathbf{X}}{n}\right)}{\lambda}\lscost{\widetilde{\mathbf{w}}^{t} - \widehat{\mathbf{w}}^{t}}{\mathbf{X}},
\end{equation*}
which is a valid bound. Then in \cite{ref:acc_ihs} the following upper bound is given without necessary justification:
\begin{equation*}
    \lscost{\widetilde{\mathbf{w}}^{t} - \widehat{\mathbf{w}}^{t}}{\mathbf{X}}\leq 2\lambda_{max}\left(\frac{\mathbf{X}^T\mathbf{X}}{n}\right)\beta^k  \lscost{\mathbf{w}^*}{2}.
\end{equation*}
to reach the final form of the error upper bound:
\begin{equation*}
    \lscost{\widetilde{\mathbf{w}}^{t+1} - \mathbf{w}^*}{\mathbf{X}} \leq \alpha^t\lscost{\mathbf{w}^*}{\mathbf{X}} + \frac{10\lambda_{max}^2\left(\frac{\mathbf{X}^T\mathbf{X}}{n}\right)}{\lambda}\beta^k  \lscost{\mathbf{w}^*}{2}.
\end{equation*}
However, this final form of the upper bound is not supported in detail as part of the presented proof. Because of the following major issue, we conclude that the proposed bound remains an unproven conjecture. The bound established for $\lscost{\widetilde{\mathbf{w}}^{t+1} - \widebar{\mathbf{w}}^{t+1}}{\mathbf{X}}$ in \eqref{wrong:inner} is used to upper bound $\lscost{\widetilde{\mathbf{w}}^{t} - \widehat{\mathbf{w}}^{t}}{\mathbf{X}}$. This is not justified as part of the proof in \cite{ref:acc_ihs}.

%% file: main.bbl
\begin{thebibliography}{10}

\bibitem{ref:bjorck}
{\AA}ke Bj{\"o}rck.
\newblock {\em Numerical methods for least squares problems}.
\newblock SIAM, 1996.

\bibitem{ref:hansen_book_2}
Per~Christian Hansen.
\newblock {\em Discrete inverse problems: insight and algorithms}, volume~7.
\newblock SIAM, 2010.

\bibitem{ref:compressed_sensing}
David~L Donoho et~al.
\newblock Compressed sensing.
\newblock {\em IEEE Trans. Inform. Theory}, 52(4):1289--1306, 2006.

\bibitem{ref:dual_rp}
Lijun Zhang, Mehrdad Mahdavi, Rong Jin, Tianbao Yang, and Shenghuo Zhu.
\newblock Recovering the optimal solution by dual random projection.
\newblock In {\em Conference on Learning Theory}, pages 135--157, 2013.

\bibitem{ref:dual_drineas}
Agniva Chowdhury, Jiasen Yang, and Petros Drineas.
\newblock An iterative, sketching-based framework for ridge regression.
\newblock In {\em International Conference on Machine Learning}, pages
  988--997, 2018.

\bibitem{ref:admm}
Stephen Boyd, Neal Parikh, Eric Chu, Borja Peleato, Jonathan Eckstein, et~al.
\newblock Distributed optimization and statistical learning via the alternating
  direction method of multipliers.
\newblock {\em Found. Trends Mach. Learn.}, 3(1):1--122, 2011.

\bibitem{ref:bubeck}
S{\'e}bastien Bubeck et~al.
\newblock Convex optimization: Algorithms and complexity.
\newblock {\em Found. Trends Mach. Learn.}, 8(3-4):231--357, 2015.

\bibitem{ref:newton_sketch}
Mert Pilanci and Martin~J Wainwright.
\newblock Newton sketch: A near linear-time optimization algorithm with
  linear-quadratic convergence.
\newblock {\em SIAM J. Optim.}, 27(1):205--245, 2017.

\bibitem{ref:vogel}
Curtis~R Vogel.
\newblock {\em Computational methods for inverse problems}, volume~23.
\newblock SIAM, 2002.

\bibitem{ref:siam_review_hybrid}
Misha~E Kilmer and Dianne~P O'Leary.
\newblock Choosing regularization parameters in iterative methods for ill-posed
  problems.
\newblock {\em SIAM J. Matrix Anal. Appl.}, 22(4):1204--1221, 2001.

\bibitem{ref:iko_ms}
Ibrahim~Kurban Ozaslan.
\newblock Fast and robust solution techniqes for large scale linear least
  squares problems.
\newblock Master's thesis, Bilkent University, 2020.

\bibitem{ref:krylov_convergence}
David~G Luenberger.
\newblock {\em Introduction to linear and nonlinear programming}, volume~28.
\newblock Addison-Wesley Reading, MA, 1973.

\bibitem{ref:distributed_pilanci}
Burak Bartan and Mert Pilanci.
\newblock Polar coded distributed matrix multiplication.
\newblock {\em arXiv preprint arXiv:1901.06811}, 2019.

\bibitem{ref:distributed_1}
Eric Jonas, Qifan Pu, Shivaram Venkataraman, Ion Stoica, and Benjamin Recht.
\newblock Occupy the cloud: Distributed computing for the 99\%.
\newblock In {\em Proceedings of the 2017 Symposium on Cloud Computing}, pages
  445--451. ACM, 2017.

\bibitem{ref:templates}
Richard Barrett, Michael~W Berry, Tony~F Chan, James Demmel, June Donato, Jack
  Dongarra, Victor Eijkhout, Roldan Pozo, Charles Romine, and Henk Van~der
  Vorst.
\newblock {\em Templates for the solution of linear systems: building blocks
  for iterative methods}, volume~43.
\newblock SIAM, 1994.

\bibitem{ref:precond}
Michele Benzi.
\newblock Preconditioning techniques for large linear systems: a survey.
\newblock {\em J. Comput. Phys.}, 182(2):418--477, 2002.

\bibitem{ref:par_cheby}
Martin~H Gutknecht and Stefan R{\"o}llin.
\newblock The chebyshev iteration revisited.
\newblock {\em Parallel Comput.}, 28(2):263--283, 2002.

\bibitem{ref:jl_lemma}
William~B Johnson and Joram Lindenstrauss.
\newblock Extensions of lipschitz mappings into a hilbert space.
\newblock {\em Contemp. Math.}, 26(189-206):1, 1984.

\bibitem{ref:randnla}
Petros Drineas and Michael~W Mahoney.
\newblock Randnla: randomized numerical linear algebra.
\newblock {\em Comm. ACM}, 59(6):80--90, 2016.

\bibitem{ref:mahoney_book}
Michael~W Mahoney et~al.
\newblock Randomized algorithms for matrices and data.
\newblock {\em Found. Trends Theor. Mach. Learn}, 3(2):123--224, 2011.

\bibitem{ref:w14_sketching_asatool}
David~P Woodruff et~al.
\newblock Sketching as a tool for numerical linear algebra.
\newblock {\em Found. Trends Theor. Comput. Sci.}, 10(1--2):1--157, 2014.

\bibitem{ref:faster_ls}
Petros Drineas, Michael~W Mahoney, S~Muthukrishnan, and Tam{\'a}s Sarl{\'o}s.
\newblock Faster least squares approximation.
\newblock {\em Numer. Math.}, 117(2):219--249, 2011.

\bibitem{ref:pilanci1}
Mert Pilanci and Martin~J Wainwright.
\newblock Randomized sketches of convex programs with sharp guarantees.
\newblock {\em IEEE Trans. Inform. Theory}, 61(9):5096--5115, 2015.

\bibitem{ref:acw16_sharper}
Haim Avron, Kenneth~L Clarkson, and David~P Woodruff.
\newblock Sharper bounds for regularized data fitting.
\newblock {\em arXiv preprint arXiv:1611.03225}, 2016.

\bibitem{ref:ihs}
Mert Pilanci and Martin~J Wainwright.
\newblock Iterative hessian sketch: Fast and accurate solution approximation
  for constrained least-squares.
\newblock {\em J. Mach. Learn. Res.}, 17(1):1842--1879, 2016.

\bibitem{ref:rokhlin}
Vladimir Rokhlin and Mark Tygert.
\newblock A fast randomized algorithm for overdetermined linear least-squares
  regression.
\newblock {\em Proc. Nat. Acad. Sci. India Sect. A}, 105(36):13212--13217,
  2008.

\bibitem{ref:blendenpik}
Haim Avron, Petar Maymounkov, and Sivan Toledo.
\newblock Blendenpik: Supercharging lapack's least-squares solver.
\newblock {\em SIAM J. Sci. Comput.}, 32(3):1217--1236, 2010.

\bibitem{ref:lsrn}
Xiangrui Meng, Michael~A Saunders, and Michael~W Mahoney.
\newblock Lsrn: A parallel iterative solver for strongly over-or
  underdetermined systems.
\newblock {\em SIAM J. Sci. Comput.}, 36(2):C95--C118, 2014.

\bibitem{ref:acc_ihs}
Jialei Wang, Jason~D Lee, Mehrdad Mahdavi, Mladen Kolar, Nathan Srebro, et~al.
\newblock Sketching meets random projection in the dual: A provable recovery
  algorithm for big and high-dimensional data.
\newblock {\em Electron. J. Stat.}, 11(2):4896--4944, 2017.

\bibitem{ref:3d_image}
M~Bertero and M~Piana.
\newblock Inverse problems in biomedical imaging: modeling and methods of
  solution.
\newblock In {\em Complex systems in biomedicine}, pages 1--33. Springer, 2006.

\bibitem{ref:mihs}
Ibrahim~K. Ozaslan, Mert Pilanci, and Orhan Arikan.
\newblock Iterative hessian sketch with momentum.
\newblock {\em 2019 IEEE International Conference on Acoustics, Speech and
  Signal Processing (ICASSP)}, 2019.

\bibitem{ref:edelman}
Alan Edelman and Yuyang Wang.
\newblock Random matrix theory and its innovative applications.
\newblock In {\em Advances in Applied Mathematics, Modeling, and Computational
  Science}, pages 91--116. Springer, 2013.

\bibitem{ref:avron_kernel}
Haim Avron, Kenneth~L Clarkson, and David~P Woodruff.
\newblock Faster kernel ridge regression using sketching and preconditioning.
\newblock {\em SIAM J. Matrix Anal. Appl.}, 38(4):1116--1138, 2017.

\bibitem{ref:cnw15_optimal_stable_rank}
Michael~B Cohen, Jelani Nelson, and David~P Woodruff.
\newblock Optimal approximate matrix product in terms of stable rank.
\newblock {\em arXiv preprint arXiv:1507.02268}, 2015.

\bibitem{ref:lsqr}
Christopher~C Paige and Michael~A Saunders.
\newblock Lsqr: An algorithm for sparse linear equations and sparse least
  squares.
\newblock {\em ACM Trans. Math. Software}, 8(1):43--71, 1982.

\bibitem{ref:polyak}
Boris~T Polyak.
\newblock Some methods of speeding up the convergence of iteration methods.
\newblock {\em Comput. Math. Math. Phys}, 4(5):1--17, 1964.

\bibitem{ref:boyd_cvx}
Stephen Boyd and Lieven Vandenberghe.
\newblock {\em Convex optimization}.
\newblock Cambridge university press, 2004.

\bibitem{ref:candes_adaptive}
Brendan O’donoghue and Emmanuel Candes.
\newblock Adaptive restart for accelerated gradient schemes.
\newblock {\em Found. Comput. Math.}, 15(3):715--732, 2015.

\bibitem{ref:kn14_sparser_jl}
Daniel~M Kane and Jelani Nelson.
\newblock Sparser johnson-lindenstrauss transforms.
\newblock {\em J. ACM.}, 61(1):4, 2014.

\bibitem{ref:nn13_osnap}
Jelani Nelson and Huy~L Nguy{\^e}n.
\newblock Osnap: Faster numerical linear algebra algorithms via sparser
  subspace embeddings.
\newblock In {\em 2013 IEEE 54th Annual Symposium on Foundations of Computer
  Science}, pages 117--126. IEEE, 2013.

\bibitem{ref:tz12}
Mikkel Thorup and Yin Zhang.
\newblock Tabulation-based 5-independent hashing with applications to linear
  probing and second moment estimation.
\newblock {\em SIAM J. Comput.}, 41(2):293--331, 2012.

\bibitem{ref:coh16_nearly_tight}
Michael~B Cohen.
\newblock Nearly tight oblivious subspace embeddings by trace inequalities.
\newblock In {\em Proceedings of the twenty-seventh annual ACM-SIAM symposium
  on Discrete algorithms}, pages 278--287. SIAM, 2016.

\bibitem{ref:nn14_lower}
Jelani Nelson and Huy~L Nguy{\AA}n.
\newblock Lower bounds for oblivious subspace embeddings.
\newblock In {\em International Colloquium on Automata, Languages, and
  Programming}, pages 883--894. Springer, 2014.

\bibitem{ref:kmn_11_almost_optimal}
Daniel Kane, Raghu Meka, and Jelani Nelson.
\newblock Almost optimal explicit johnson-lindenstrauss families.
\newblock In {\em Approximation, Randomization, and Combinatorial Optimization.
  Algorithms and Techniques}, pages 628--639. Springer, 2011.

\bibitem{ref:hansen_book1}
Per~Christian Hansen.
\newblock {\em Rank-deficient and discrete ill-posed problems: numerical
  aspects of linear inversion}, volume~4.
\newblock SIAM, 2005.

\bibitem{ref:nocedal}
Jorge Nocedal and Stephen Wright.
\newblock {\em Numerical optimization}.
\newblock Springer Science \& Business Media, 2006.

\bibitem{ref:forcing_term}
Stanley~C Eisenstat and Homer~F Walker.
\newblock Choosing the forcing terms in an inexact newton method.
\newblock {\em SIAM J. Sci. Comput.}, 17(1):16--32, 1996.

\bibitem{ref:inexact_survey}
Stephen~G Nash.
\newblock A survey of truncated-newton methods.
\newblock {\em J. Comput. Appl. Math.}, 124(1-2):45--59, 2000.

\bibitem{ref:nocedal2}
Albert~S Berahas, Raghu Bollapragada, and Jorge Nocedal.
\newblock An investigation of newton-sketch and subsampled newton methods.
\newblock {\em arXiv preprint arXiv:1705.06211}, 2017.

\bibitem{ref:avron_hutchinson}
Haim Avron and Sivan Toledo.
\newblock Randomized algorithms for estimating the trace of an implicit
  symmetric positive semi-definite matrix.
\newblock {\em J. ACM}, 58(2):8, 2011.

\bibitem{ref:ir_tool}
Silvia Gazzola, Per~Christian Hansen, and James~G Nagy.
\newblock Ir tools: a matlab package of iterative regularization methods and
  large-scale test problems.
\newblock {\em Numer. Algorithms}, 81(3):773--811, 2019.

\bibitem{ref:elden_shift}
Lars Elden.
\newblock Algorithms for the regularization of ill-conditioned least squares
  problems.
\newblock {\em BIT}, 17(2):134--145, 1977.

\bibitem{ref:symmetric_systems}
Horst~D Simon.
\newblock Analysis of the symmetric lanczos algorithm with reorthogonalization
  methods.
\newblock {\em Linear algebra and its applications}, 61:101--131, 1984.

\bibitem{ref:gazzola}
Silvia Gazzola, Paolo Novati, and Maria~Rosaria Russo.
\newblock On krylov projection methods and tikhonov regularization.
\newblock {\em Electron. Trans. Numer. Anal}, 44(1):83--123, 2015.

\bibitem{ref:hansen_matlab}
Per~Christian Hansen.
\newblock Regularization tools: a matlab package for analysis and solution of
  discrete ill-posed problems.
\newblock {\em Numer. Algorithms}, 6(1):1--35, 1994.

\bibitem{ref:tum}
Raphael Hunger.
\newblock {\em Floating point operations in matrix-vector calculus}.
\newblock Technical University of Munich, 2007.

\end{thebibliography}
